\documentclass[12pt,oneside,english,draft]{amsart}
\usepackage[T1]{fontenc}
\usepackage[latin9]{inputenc}
\usepackage{geometry}
\usepackage{mathrsfs}
\usepackage{amsthm}
\usepackage{amstext}
\usepackage{amssymb}
\usepackage{esint}

\makeatletter
\numberwithin{equation}{section}
\numberwithin{figure}{section}
\theoremstyle{plain}
\newtheorem{thm}{\protect\theoremname}[section]
  \theoremstyle{definition}
  \newtheorem{defn}[thm]{\protect\definitionname}
  \theoremstyle{plain}
  \newtheorem{prop}[thm]{\protect\propositionname}
  \theoremstyle{plain}
  \newtheorem{cor}[thm]{\protect\corollaryname}
  \theoremstyle{plain}
  \newtheorem{lem}[thm]{\protect\lemmaname}

\makeatother

\usepackage{babel}
  \providecommand{\corollaryname}{Corollary}
  \providecommand{\definitionname}{Definition}
  \providecommand{\lemmaname}{Lemma}
  \providecommand{\propositionname}{Proposition}
\providecommand{\theoremname}{Theorem}

\begin{document}

\title{Free Monotone Transport}

\author{A. Guionnet$^{\dagger}$ and D. Shlyakhtenko$^{\ddagger}$}

\thanks{$^{\dagger}$guionnet@math.mit.edu, CNRS \& \'Ecole Normale Sup\'eerieure
de Lyon, France \mbox{ and } MIT, Department of mathematics, USA. Research supported by ANR-08-BLAN-0311-01 and Simons foundation.}

\thanks{$^{\ddagger}$shlyakht@math.ucla.edu, UCLA Department of Mathematics.
Research supported by NSF grants DMS-0900776 and DMS-1161411 and DARPA HR0011-12-1-0009}
\begin{abstract}
By solving a free analog of the Monge-Ampère equation, we prove a
non-commutative analog of Brenier's monotone transport theorem: if
an $n$-tuple of self-adjoint non-commutative random variables $Z_{1},\dots,Z_{n}$
satisfies a regularity condition (its conjugate variables $\xi_{1},\dots,\xi_{n}$
should be analytic in $Z_{1},\dots,Z_{n}$ and $\xi_{j}$ should be
close to $Z_{j}$ in a certain analytic norm), then there exist invertible
non-commutative functions $F_{j}$ of an $n$-tuple of semicircular
variables $S_{1},\dots,S_{n}$, so that $Z_{j}=F_{j}(S_{1},\dots,S_{n})$.
Moreover, $F_{j}$ can be chosen to be monotone, in the sense that
$F_{j}=\mathscr{D}_{j}g$ and $g$ is a non-commutative function with
a positive definite Hessian. In particular, we can deduce that $C^{*}(Z_{1},\dots,Z_{n})\cong C^{*}(S_{1},\dots,S_{n})$
and $W^{*}(Z_{1},\dots,Z_{n})\cong L(\mathbb{F}(n))$. Thus our condition
is a useful way to recognize when an $n$-tuple of operators generate
a free group factor. We obtain as a consequence that the $q$-deformed
free group factors $\Gamma_{q}(\mathbb{R}^{n})$ are isomorphic (for
sufficiently small $q$, with bound depending on $n$) to free group
factors. We also partially prove a conjecture of Voiculescu by showing
that free Gibbs states which are small perturbations of a semicircle
law generate free group factors. Lastly, we show that entrywise monotone
transport maps for certain Gibbs measure on matrices are well-approximated
by the matricial transport maps given by free monotone transport.
\end{abstract}
\maketitle
\section{Introduction.}

\subsection{On a notion of density in free probability.}

Let $X_{1},\dots,X_{n}$ be classical random variables. Thus $X=(X_{1},\dots,X_{n})$
can be viewed as a function  %
defined on a measure space $(\Omega,\omega)$
with values in (say) $\mathbb{R}^{n}$. Of special interest is the
law of $X$, which is the measure on $\mathbb{R}^{n}$ obtained as
the push-forward $\mu=X_{{{*}}}\omega$. Very often, it is assumed that
$\mu$ is Lebesgue absolutely continuous and the density $\rho=d\mu/(\prod dx_{j})$
then plays a key role. For example, the density is involved in the
classical definition of entropy ($\int\rho\log\rho\ \prod dx_{j}$),
Fisher information ($\int\left|\frac{\rho'}{\rho}\right|^{2}\prod dx_{j}$)
and so on.

In passing to the non-commutative case, one assumes that $X_{1},\dots,X_{n}$
are self-adjoint elements of some finite von Neumann algebra $(M,\tau)$,
where $\tau:M\to\mathbb{C}$ is a normal faithful trace. In other
words, $X_{j}$ are self-adjoint operators on some Hilbert space $H$
containing %
a vector $\xi$ so that $\tau(T)=\langle T\xi,\xi\rangle$
satisfies $\tau(P(X)Q(X))=\tau(Q(X)P(X))$ for any non-commutative
polynomials $P$ and $Q$ evaluated at $X=(X_{1},\dots,X_{n})$. %

The only (partially) satisfactory extension of the notion of %
 joint
law of $X_{1},\dots,X_{n}$ to the non-commutative case uses the moment
method. More precisely, one defines the non-commutative law of $X_{1},\dots,X_{n}$
as the functional $\mu$ which assigns to a non-commutative monomial
$P$ the value $\tau(P(X))$.

Unfortunately, no satisfactory replacement notion of \emph{density}
has been obtained so far with the exception of the case $n=1$. In this case, %
spectral theory gives a suitable replacement; however, in absence
of any commutation between $X_{1},\dots,X_{n}$ there is no satisfactory
theory of ``joint spectrum'' even if $X_{j}$'s all act on a finite-dimensional
vector space. This causes a number of problems %
 in free probability
theory. For example, Voiculescu introduced two definitions of free
entropy ($\chi$ and $\chi^{*}$), but for neither of them is there
a simple formula analogous to the classical case as soon as $n\neq1$. %

\subsection{Free Gibbs states and log-concave measures.}

One hint that gives hope that a satisfactory replacement for the notion
of density can be found lies in the existence of free analogs of strictly log-concave
measures. Log-concave measures on $\mathbb{R}^{n}$ are probability
laws having particularly nice density: it has the form $\exp(-V(x_{1},\dots,x_{n}))dx_{1}\cdots dx_{n}$,
where $V$ is a strictly convex function. 

It turns out that these laws have free probability analogs: If  $V$
is close to a quadratic potential (that  is $V(X_{1},\dots,X_{n})=\frac{1}{2}\sum X_{j}^{2}+\beta W(X_{1},\dots,X_{n})$
for a fixed 
polynomial $W$ and sufficiently small $\beta$), there
exists a unique non-commutative law $\tau_{V}$ which satisfies the
``Schwinger-Dyson'' equation \cite{guionnet:icm,guionnet-edouard:combRM}
\[
\tau_{V}(P\mathscr{D}V)=\tau_{V}\otimes\tau_{V}(Tr(\mathscr{J}P)).
\]
This law $\tau_{V}$ is called the \emph{free Gibbs law with potential
$V$. }Here $\mathscr{D}$ and $\mathscr{J}$ are suitable non-commutative
replacements for the gradient and Jacobian, respectively. 

Free Gibbs laws associated to such convex non-commutative functions are very well-behaved. %
 It is worth mentioning that
they arise as limits of random matrix models associated to probability
measures with log-concave density $\exp(-NTr(V(A_{1},\dots,A_{n})))$
on the spaces of $N\times N$ self-adjoint matrices. Alternatively,
the law $\tau_{V}$ can be characterized as the minimizer of the relative
free entropy
\[
\chi_{V}(\tau_{V})=\chi(\tau_{V})-\tau_{V}(V),
\]
where $\chi(\tau_{V})$ is Voiculescu's (microstates) free entropy
\cite{dvv:entropy2,dvv:entropysurvey}. 

Free Gibbs states have many nice properties with respect to free stochastic
calculus, various free differential operators and so on (see e.g.
\cite{guionnet-edouard:combRM,guionnet-segala:secondOrder,alice-shlyakhtenko-freeDiffusions}).
Thus the fact that, without knowing what a non-commutative density %
is,  we are able to single out a class of non-commutative laws which
are similar in property to classical log-concave measures (which by
definition are measures with a nice density!) strongly suggests that
a non-commutative version of density could be found. Moreover, if
it were to be found, such log-concave measures are natural first candidates
for a detailed study.

\subsection{Monotone transportation maps as replacements for densities.}

There is an alternative way of talking about densities in the classical
case, which is a consequence of the following theorem of Brenier:
assume that $\mu$ is a measure on $\mathbb{R}^{n}$ satisfying some
technical conditions (Lebesgue absolutely continuous, finite second
moment, etc.) Let $\nu$ denote the standard %
 Gaussian measure on $\mathbb{R}^{n}$
. Then \emph{there exists a canonical monotone transport map $\phi_{\mu}:\mathbb{R}^{n}\to\mathbb{R}^{n}$
from $\nu$ to $\mu$. }Here by a \emph{transport map from $\nu$
to $\mu$ }we mean a map $\phi$ satisfying $\phi_{{{*}}}\nu=\mu$; and
$\phi$ is called \emph{monotone} if $\phi=\nabla\psi$ for some convex
function $\psi$ (roughly speaking, this means that the Jacobian of
$\phi$ is positive-definite almost everywhere). Our point is that
such a transport map carries all the information contained in the
knowledge of the density $d\mu/\prod dx_{j}$ (indeed, this density
is essentially the Jacobian of the map). A %
 sign of the fact
that $\phi_{\mu}$ is a good analytic object is that it continues
to exist even when the density does not (for example, when $\mu$
is not Lebesgue absolutely continuous, but does not give mass to ``small
sets'', cf. \cite{brennier:polarFact}).

It is thus very tempting to ask if there is any analog of Brenier's
results in the case of non-commutative probability. Here one immediately
faces a major obstacle: unlike in the classical case, there are many
more non-commutative probability spaces than classical probability
spaces. Indeed, up to isomorphism, the unit interval $[0,1]$ with
Lebesgue measure is the unique non-atomic probability space. At the
same time, there are many non-isomorphic non-commutative von Neumann
algebras. In fact, there is not even a (separable) von Neumann algebra
that contains all others \cite{ozawa:noUniversal}, thus there cannot
exist a ``master law'' so that all laws can be obtained as ``push-forwards''
of that law (as is the case classically). However, as we show in the
present paper, one may hope to obtain such a result in certain cases
(and, indeed, the Brenier map only exists under certain analytic assumptions
about the target measure, even in the classical case).

\subsection{The main results.}

In this paper, we give the first examples of existence of non-commutative
monotone transport. We thus take the first steps in the study of this
subject.

To state our results, recall that the Gaussian law $\nu$ can be characterized
by the following integration by parts formula: for any $f:\mathbb{R}^{n}\to\mathbb{R}^{n}$,
denoting by $x$ the vector $(x_{1},\dots,x_{n})$, 
\[
\int x\cdot f(x)\ d\nu(x)=\int Tr(Jf(x))d\nu(x),
\]
where $Jf$ denotes the Jacobian of $f$. In other words, $x=J^{*}I$,
where $J^{*}$ denotes the adjoint of $J$ viewed as an unbounded
operator $L^{2}(\mathbb{R}^{n},\nu)\to L^{2}(M_{n\times n},\nu)$,
and $I$ denotes the $n\times n$ identity matrix.

Recall that a free semicircular family $S=(S_{1},\dots,S_{n})$ %
 is a free-probability
analog of an iid $n$-tuple of Gaussian random variables. The semicircle
law can also be characterized by the formula of the form
\[
S=\mathscr{J}^{*}(I),
\]
where $\mathscr{J}^{*}$ is the adjoint of the map $\mathscr{J}$,
which is a suitable non-commutative replacement for the Jacobian. 

Our main theorem shows that monotone transport exists in the non-commutative
case, provided that the law of an $n$-tuple $Z=(Z_{1},\dots,Z_{n})$
is ``close'' to the semicircle law, in the sense that it ``almost
satisfies'' the equation $Z=\mathscr{J}^{*}(I)$, i.e., $Z$ is close
to $\mathscr{J}^{*}(I)$ in a certain analytic function norm $\Vert\cdot\Vert_{A}$
(see \S\ref{sub:powerSeriesAndNorm}). The conditions on $Z$ automatically
imply that the law of $Z$ is a free Gibbs law.
\begin{thm}
[Existence of monotone transport] Let $S=(S_{1},\dots,S_{n})$ be
a free semicircular family. If $Z_{1},\dots,Z_{n}\in(M,\tau)$ are
non-commutative random variables  such that there exists $F=F^*$ in the closure of $\mathcal A$ for the norm
$\|.\|_A$  so that $\mathscr{J}^{*}(I)={\mathscr{D}}F$  is so that  %
$\Vert Z-\mathscr{J}^{*}(I)\Vert_{A}$
is sufficiently small, then %
 there exist an $Y_{1},\dots,Y_{n}\in W^{*}(S_{1},\dots,S_{n})$
so that:
\begin{itemize}
\item the law of $Y=(Y_{1},\dots,Y_{n})$ is the same as the law of $Z=(Z_{1},\dots,Z_{n})$,
so that $W^{*}(Z)\cong W^{*}(Y)$;
\item $W^{*}(S)=W^{*}(Y)\ (\cong W^{*}(Z))$, so that $Y_{j}$ can be viewed
as ``non-commutative measurable functions'' of $S_{1},\dots,S_{n}$
(we in fact show that these functions can be taken to be non-commutative
power series);
\item $Y_{j}=\mathscr{D}_{j}G$ for some $G\in W^{*}(S_{1},\dots,S_{n})$
and $\mathscr{J}Y$ is positive-definite.
\end{itemize}
\end{thm}
Here $\mathscr{D}$ and $\mathscr{J}$ are suitable non-commutative
replacements for the gradient and Jacobian of a function; these differential
operators play key roles in free probability theory \cite{dvv:entropy5,dvv:entropysurvey,dvv:cyclomorphy,alice:StFlour,guionnet:icm,guionnet-anderson-zeitouni}).
In particular, the last condition is a kind of convexity requirement
on $G$ (since $\mathscr{J}Y=\mathscr{J}\mathscr{D}G$ is a kind of
Hessian of $G$).

By analogy with Brenier's theorem, we call $F=\mathscr{D}G$ ``free
monotone transport'' from the semicircle law to the law of $Z$. %
We can view the function $F$ as playing the role of density; the
fact that we are able to construct $F$ in the space of certain analytic
maps is a reflection of the fact that the law of $Z_{1},\dots,Z_{n}$
has a ``nice density'' relative to the free semicircle law. 

By Voiculescu's change of variables formula, one can claim that for
laws $\mu$ satisfying the assumptions of our theorem,
\[
\chi(\mu)=\tau\otimes\tau\otimes\operatorname{Tr}(\log\mathscr{J}F_{\mu})+\textrm{universal constant.}
\]
We thus see that free entropy can be expressed by a simple formula
in terms of the monotone transport map, and is concave in this map.

Our theorem yields a number of isomorphism results, since it gives
a rather general condition of when an $n$-tuple of operators generates
a free group factor. In particular, the theorem applies to free log-concave
Gibbs laws with potential of the form $\frac{1}{2}\sum X_{j}^{2}+\beta W(X_{1},\dots,X_{n})$
if $\beta$ is sufficiently small:
\begin{thm}
Let $W$ be a fixed self-adjoint %
polynomial in $n$ variables and set 

\noindent
$V_{\beta}=\frac{1}{2}\sum_{j=1}^{n}X_{j}^{2}+\beta W$.
Let $\tau_{\beta}$ be the free Gibbs state with potential $V_{\beta}$.
Then for sufficiently small $\beta$, $W^{*}(\tau_{\beta})\cong L(\mathbb{F}_{n})$.
\end{thm}
This theorem partially answers in the affirmative a conjecture of
Voiculescu \cite[p. 240]{voiculescu:conjectureAboutPotentials} (the
full conjecture involves arbitrary values of $\beta$).

We also prove:
\begin{thm}
Let $\Gamma_{q}(\mathbb{R}^{n})$ be the von Neumann algebra generated
by $n$ $q$-deformed semicircular elements of Bozejko and Speicher.
Then there are numbers $q_{0}(n)>0$ so that for all $|q|<q_{0}(n)$,
$\Gamma_{q}(\mathbb{R}^{n})\cong L(\mathbb{F}_{n})$.
\end{thm}
This settles (for small values of $q$) the 20-year old question of
the isomorphism class of such $q$-deformed von Neumann algebras.
Furthermore, since $L(\mathbb{F}_{n})$ is a factor, is non-hyperfinite,
strongly solid \cite{popa-ozawa:stronglySolid} (thus is solid and
has no Cartan subagebras), has the Haagerup property and the complete
metric approximation property, (again, for the small values of $q$)
our result can be viewed as the culmination of a number of results
about these algebras, including \cite{Speicher:q-comm,speicher:q-GaussianProcesses,nou:qNonInjective,sniady:q-RM-model,sniady:q-nonGamma,shlyakht:someEstimates,ricard:QFactoriality,shlyakht:lowerEstimates,kennedy-nica:exactQDeformed,dabrowski:SPDE,avsec:Q}.

It would be very interesting to extend our main theorem (which is
limited by its hypothesis to laws ``close'' to the semicircle law),
and to study the analytical properties of the map $F_{\mu}$ (these
would be non-commutative analogs of properties of the density associated
to a classical law $\mu$). For example, $F_{\mu}$ could be used
to at least formulate multi-variables analogs of various regularity
results in single-variable free probability. Similarly, $F_{\mu}$
could play a prominent role in the study of multi-variable random
matrix ensembles.

\subsection{Free monotone transport as a limit of classical monotone transport.}

As we mentioned before, the free Gibbs law $\tau_{V}$ can be obtained
as a limit of classical Gibbs measures in the following sense: for
any non-commutative polynomial $P$,
\[
\tau_{V}(P)=\lim_{N\to\infty}\mathbb{E}_{\mu_{V}^{(N)}}\left[\frac{1}{N}Tr(P(A_{1},\dots,A_{n}))\right]
\]
where $\mu_{V}^{(N)}$ is a measure on $n$-tuples of self-adjoint
$N\times N$ matrices given by
\[
d\mu_{V}^{(N)}=\exp(-NTrV(A_{1},\dots,A_{n}))\times\textrm{Lebesgue measure.}
\]
Let $W$ be a fixed self-adjoint  %
non-commutative polynomial and set $V=\frac{1}{2}\sum X_{j}^{2}+W$.
If we denote by $\mu^{(N)}$ the Gaussian measure, there exists a
unique monotone transport map $F^{(N)}$ which pushes forward $\mu^{(N)}$
into $\mu_{V}^{(N)}$. This map operates on matrices entrywise. On
the other hand, let $F$ be the free monotone transport map taking
the semicircle law to the free Gibbs state $\tau_{V}$. Using functional
calculus, $F$ induces for each $N$ a map of $n$-tuples of $N\times N$
matrices (a kind of ``matricial'' transport). In Theorem \ref{thrm:MatricialTransport},
we show that $F^{(N)}- F\to 0$ in $L^{2}$. %
In other words, the ``entrywise''
transport %
 is well-approximated by ``matricial'' transport
given by functional calculus on matrices.  This is somewhat parallel to Biane's 
result \cite{biane:SBTransform} which shows that in the Gaussian case, the ``entrywise'' Segal--Bargmann transform is asymptotically
``matricial'' and yields the limiting transform in free probability theory.

\subsection{The one-dimensional case.}

In the case that $n=1$, so that we are dealing with a single random
variable, non-commutativity goes away. In this section, we explain
the proof of our main result in the case $n=1$; indeed, many of the
formulae that hold in the general multi-variable case have simple
ad hoc explanations in the one-variable case. The arguments in this
section are subsumed by the arguments for a general $n$; however,
we feel that their inclusion is justified as they serve to clarify
our proof of the general case. 

We thus start with a semicircular variable $X$ %
 viewed as the
operator of multiplication by $x$ on $L^{2}([-2,2],\eta)$, where
$\eta$ is the semicircle law: $d\eta(x)=\chi_{[-2,2]}(x)\cdot\frac{1}{2\pi}\sqrt{4-x^{2}}dx$.
We write $\tau(h)=\int h(x)d\eta(x)$. Thus $\eta$ is the unique
maximizer %
of the functional
\begin{equation}\label{tyr}
\chi(\mu)=\iint\log|s-t|d\mu(s)d\mu(t)-\int\frac{1}{2}t^{2}d\mu(t).
\end{equation}
We fix a function $W$ which is analytic on the disk of radius $A$
and let $V(x)=\frac{1}{2}x^{2}+W(x)$. We now consider measure $\eta_{V}$
which is the unique maximizer of
\[
\chi_{V}(\mu)=\iint\log|s-t|d\mu(s)d\mu(t)-\int V(t)d\mu(t).
\]
If we write $\tau_{V}$ for the linear functional $\tau_{V}(h)=\int h(x)d\eta_{V}(x)$,
then $\tau_{V}$ is the free Gibbs state associated to $V$. 

Our goal is to construct a function $F$ which is analytic on a disk
of radius $A'$ for some %
 $4<A'<A$, so that $Y=F(X)$ has
law $\eta_{V}$ when $X$ has law $\eta$. %
 In other words,
\begin{equation}
\int h(F(x))d\eta(x)=\int h(y)d\eta_{V}(y),\label{eq:oneVarWhatWeNeed}
\end{equation}
or, equivalently, $F_{*}\eta=\eta_{V}$. 

By definition $\chi_V((x+\delta f)_{*}\eta_V)\le\chi_V(\eta_V)$ for all  real numbers $\delta$ %
which implies that the  unique minimizer $\eta_{V}$ satisfies the Schwinger-Dyson equation
\[
\int V'(t)f(t)d\eta_{V}(t)=\iint\frac{f(s)-f(t)}{s-t}d\eta_{V}(s)d\eta_{V}(t).
\]
If we write $\mathscr{J}f(s,t)=\frac{f(s)-f(t)}{s-t}$, $\mathscr{D}V(x)=V'(x)$
and identify $L^{\infty}([-2,2]^{2},d\eta_{V}(s)d\eta_{V}(t))$ with
$L^{\infty}([-2,2],d\eta_{V})\bar{\otimes}L^{\infty}([-2,2],d\eta_{V})$,
then the Schwinger-Dyson equation reads
\begin{equation}
\tau_{V}(f\mathscr{D}V)=\tau_{V}\otimes\tau_{V}(\mathscr{J}f).\label{eq:oneVarSD}
\end{equation}
In the case that $\tau_V$ has a connected support ( e.g., $V$ is strictly %
 convex on a sufficiently large interval or 
if $W$ is sufficiently small), the Schwinger-Dyson equation determines
$\tau_{V}$ (or, equivalently, $\eta_{V}$) uniquely, see e.g.  \cite{BG13}. 

Note that (\ref{eq:oneVarSD}) is equivalent to saying that $\mathscr{J}^{*}(1\otimes1)=\mathscr{D}V$.
It is not hard to see that $\mathscr{J}^{*}(1\otimes1)=2\int\frac{1}{x-y}d\eta_{V}(y)$
(here and for the remainder of the section, the improper integrals
are taken in the sense of principal value). Thus (\ref{eq:oneVarWhatWeNeed})
is equivalent to 
\[
2\int\frac{1}{x-y}d\eta_{V}(y)=V'(x)\quad \forall x\in \mbox{supp}(\eta_V)\,. %
\]
Replacing $x$ by $F(x)$ and $y$ by $F(y)$ and remembering that
$F_{{*}}\eta=\eta_{V}$ gives us the following equation for $F$:
\[
2\int\frac{1}{F(x)-F(y)}d\eta(y)=V'(F(x))=F(x)+W'(F(x)), \quad \forall x\in [-2,2]\, .%
\]
To solve this equation, we try to find a solution of the form $F(x)=x+f(x)$ with $f$ small if $W$ is. %
Then the equation becomes
\begin{equation}
2\int\frac{1}{x-y}\frac{1}{1+\mathscr{J}f(x,y)}d\eta(y)=x+f(x)+W'(x+f(x)),\label{eq:oneVarEquivalentEquation1}
\end{equation}
where, as before, we use the notation $\mathscr{J}f(x,y)=\frac{f(x)-f(y)}{x-y}$. 

Using the fact that $\eta$ is the semicircle law, so that $2\int\frac{1}{x-y}d\eta(x)=x$,
we find that (\ref{eq:oneVarEquivalentEquation1}) is equivalent to
\[
-2\int\frac{1}{x-y}\frac{\mathscr{J}f(x,y)}{1+\mathscr{J}f(x,y)}d\eta(y)=f(x)+W'(x+f(x)).
\]
To deal with the difference quotient $\mathscr{J}f(x,y)=\frac{f(x)-f(y)}{x-y}$ %
, we add and subtract the term $f'(x)$:
\begin{multline*}
2\int\frac{1}{x-y}(f'(x)-\mathscr{J}f(x,y))\frac{1}{1+\mathscr{J}f(x,y)}d\eta(y)\\
=f'(x)\left[2\int\frac{1}{x-y}\frac{1}{1+\mathscr{J}f(x,y)}d\eta(x)\right]+f(x)+W'(x+f(x)).
\end{multline*}
We can substitute (\ref{eq:oneVarEquivalentEquation1}) into the right
hand side, to get
\begin{align}
 & 2\int\frac{1}{x-y}(f'(x)-\mathscr{J}f(x,y))\frac{1}{1+\mathscr{J}f(x,y)}d\eta(y)\label{eq:oneVarEquation}\\
 & \qquad\qquad\qquad=f'(x)\left[x+f(x)+W'(x+f(x))\right]+f(x)+W'(x+f(x))\nonumber \\
 & \qquad\qquad\qquad=(xf)'(x)+\frac{1}{2}(f(x)^{2})'+[W(x+f(x))]'.\nonumber 
\end{align}
The resulting equation is equivalent to the previous one under the
assumption that $f'(x)\neq -1$ %
almost surely. 

Since the right-hand side is a full derivative, it is natural to try
to express the left-hand side as a full derivative, too.
Simple algebra shows that 

\[
\int\frac{1}{x-y}(f'(x)-\mathscr{J}f(x,y))\frac{1}{1+\mathscr{J}f(x,y)}d\eta(y)\\
=\partial_{x}\int\log(1+\mathscr{J}f(x,y))d\eta(y).
\] 
Substituting this into (\ref{eq:oneVarEquation}) and removing derivatives
from both sides finally gives us
\[
2\int\log(1+\mathscr{J}f(x,y))d\eta(y)=(xf)(x)+\frac{1}{2}(f(x))^{2}+W(x+f(x))
+\textrm{const.}
\]
Finally, we seek $g$ so that $f=g'$ solves the previous equation.
This gives us the following equation for $g$:
\begin{equation}
xg'(x)=-W(x+g'(x))-\frac{1}{2}(g'(x))^{2}+2\int\log\left(1+\frac{g'(x)-g'(y)}{x-y}\right)d\eta(y)+\textrm{const.}\label{eq:MongeAmpereKind}
\end{equation}
The constant can be fixed by requiring that both sides of the equation
vanish at $x=0$. Note that the operator $\mathscr{N}g=xg'$ multiplies
monomials of degree $n$ by $n$. Let $\Sigma$ be the inverse of
$\mathscr{N}$ (defined on polynomials with zero constant term), %
 $\Sigma(\hat g)(x)=\int_0^x y^{-1} \hat g(y) dy$, %
 and let
$\hat{g}=\mathscr{N}g$ so that $g=\Sigma\hat{g}$. With this notation,
our equation becomes
\[
\hat{g}(x)
=-W(x+(\Sigma\hat{g})'(x))-\frac{1}{2}((\Sigma\hat{g})(x))^{2}+2\int\log\left(1+\frac{(\Sigma\hat{g})'(x)-(\Sigma\hat{g})'(y)}{x-y}\right)d\eta(y)+\textrm{const.}
\]
We can now rewrite this equation as a fixed point equation
\[
\hat{g}(x)=\Psi_x(\hat{g}(x)). %
\]
Note that $\Psi_x(0)=W(x)$. For $W$ small, this equation can be solved
by iteration in the space of analytic functions converging on a large
enough disk. The essential fact here is that $\hat{g}\mapsto\Psi_x(\hat{g})$
is Lipschitz for a certain analytic norm for any $x$. %
 This gives us a function
$\hat{g}$, which is sufficiently small. Retracing our steps (and
noting that for sufficiently small $W$, $f$ is small, so that $|f'(x)|<1$
almost surely), we get back the desired map $F$ given by $F=x+f=G'$ if %
$G=\frac{1}{2}x^{2}+g(x)$.

It is rather straightforward that if we consider $V_{\beta}=\frac{1}{2}x^{2}+W_{\beta}(x)$
for some family $W_{\beta}$ analytically depending on $\beta$ and
sufficiently small, then our transport map $F=F_{V_{\beta}}$ also
depends on $\beta$ analytically. 

It is worth noting that for $W$ sufficiently small, the map $F(x)=x+f(x)$ %
has a positive derivative and is thus monotone. It is therefore the
unique monotone map satisfying $F_{*}\eta=\eta_{V}$.
In particular,
the map $F$ gives optimal transport between $\eta$ and $\eta_{V}$. 

It was pointed out to us by Y. Dabrowski that equation (\ref{eq:MongeAmpereKind})
can be regarded as a free analog of the classical Monge-Ampère equation.
To see this, put $V(x)=\frac{1}{2}x^{2}$ and $\tilde{V}=V+W$, and
consider the classical Monge-Ampère equation for the monotone transport
map $H$ from the Gibbs measure $d\mu=\exp(-V(x))dx$ to $d\tilde{\mu}=\exp(-\tilde{V}(x))dx$.
Let $JH$ stand for the Jacobian (derivative) of $H$. Then the Monge-Ampère
equation reads: %
\[
\det(JH(x))=\frac{\exp(-V(x))}{\exp(-\tilde{V}(H(x)))}=\exp(\tilde{V}(H(x))-V(x)).
\]
(We write $\det JH$ even though $JH$ is a $1\times1$ matrix to
make the analogy clearer). Taking logarithm of both sides and using
that $\log\det A=Tr(\log A)$ if $A\geq0$ gives
\[
Tr(\log JH)=\tilde{V}(H(x))-V(x). %
\]
This is quite reminiscent of (\ref{eq:MongeAmpereKind}), if we rewrite
it by setting $f=g'$, $F(x)=x+f(x)$ and $\mathscr{J}F=\frac{F(x)-F(y)}{x-y}=1+\mathscr{J}f$:
\[
2\int\log\left(\mathscr{J}F\right)d\eta(y)=\left\{ \frac{1}{2}(F(x))^{2}+W(F(x))\right\} -\frac{1}{2}x^{2}=\tilde{V}(F(x))-V(x). %
\]
Later in the paper, we also consider the $n$-variable version of
(\ref{eq:MongeAmpereKind}), which has the form
\[
(1\otimes\tau+\tau\otimes1)Tr\log\mathscr{J}F=\mathscr{S}\left[\left\{ \frac{1}{2}\sum F(X)_{j}^{2}+W(F(X))\right\} -\frac{1}{2}\sum X_{j}^{2}\right]
\]%
where $\mathscr{S}$ is a certain symmetrization operator.

\subsection{Organization of paper.}

To prove our theorem, we consider in section \S\ref{sec:Notation}
a certain space of analytic functions in several non-commuting variables.
This is  the %
 space in which our construction of the monotone transport
map takes place. We also discuss the various differential operators
which are the non-commutative replacements for gradients and Jacobians.
We finish the section with a kind of implicit function theorem for
non-commutative analytic functions.

The next section is devoted to the construction of a non-commutative
monotone transport map from the semicircle law $\tau$ to the law
$\tau_{V}$ satisfying the Schwinger-Dyson equation with $V=\frac{1}{2}\sum X_{j}^{2}+W$.
We assume that $W$ is a given non-commutative analytic function (which
we assume to be sufficiently small in a certain norm).

The next section, \S\ref{sec:Applications.}, is devoted to applications.
Finally, the last section collects some open questions and problems.

\subsection*{Acknowledgments. }

The authors are grateful to Yoann Dabrowski for many useful comments
and discussions. We also wish to mention that the idea of looking
for a free analog of \emph{optimal} transport (although via some duality
arguments) was considered some 10 years ago by Cédric Villani and
the authors; although the precise connection between \emph{optimal
}and \emph{monotone} transport is still missing for the moment, such
considerations have been an inspiration for the present work. We also thank the
anonymous referees for their numerous comments which helped us to greatly improve 
our article. %

\section{Notation and an Implicit Function Theorem.\label{sec:Notation}}

\subsection{Non-commutative polynomials, power series and norms $\Vert\cdot\Vert_{A}$.\label{sub:powerSeriesAndNorm}}

We will denote by $\mathscr{A}=\mathbb{C}\langle X_{1},\dots,X_{n}\rangle$
the algebra of non-commutative polynomials in $n$ variables. Let
$\mathscr{A}_{0}\subset\mathscr{A}$ be the linear span of polynomials
with zero constant term. %

Following \cite{guionnet-segala:secondOrder}, we consider on this
space the family of norms $\Vert\cdot\Vert_{A}$, $A>1$, defined
as follows. For a monomial $q$ and arbitrary $P\in\mathbb{C}\langle X_{1},\dots,X_{n}\rangle$,
let $\lambda_{q}(P)$ be the coefficient of $q$ in the decomposition
of $P$ as monomials; thus $P=\sum_{q}\lambda_{q}(P)q$. Then we set
\[
\Vert P\Vert_{A}=\sum_{q:\deg q\geq0}|\lambda_{q}(P)|A^{\deg q}.
\]

We'll denote by $\mathscr{A}^{(A)}$ (resp., $\mathscr{A}_{0}^{(A)}$)
the completion of $\mathscr{A}$ (resp., $\mathscr{A}_{0}$) with
respect to the norm $\Vert\cdot\Vert_{A}$. This is a Banach algebra,
and can be viewed as the algebra of absolutely convergent power series
with radius of convergence at least $A$.

Note that the norm $\Vert\cdot\Vert_{A}$ has the following property:
whenever $T_{1},\dots,T_{n}$ are elements of some Banach algebra
$Q$ and $\Vert T_{j}\Vert_{Q}\leq A$, there exists a contractive
map $\mathscr{A}^{(A)}\to Q$ sending $X_{j}$ to $T_{j}$.  Hence, in the following we
will see variables as elements of the set of polynomials $\mathscr{A}^{(A)}$. However,
$\mathscr{A}^{(A)}$ will still also denote  sets of maps, namely the completion of polynomials in the norm $\|\cdot\|_A$.

$\mathcal{A}^{op}$  is the opposite algebra, with multiplication defined by $ a^{op} \cdot b^{op} = (ba)^{op}$. It is equipped with the same norm $\|\cdot\|_A$.

\subsection{The operators $\mathscr{N}$, $\Sigma$, %
$\Pi$, $\mathscr{J}$, $\partial$  and $\mathscr{D}$.}\label{secdef}
Let us denote by $\mathscr{N}$ the linear operator on $\mathscr{A}=\mathbb{C}\langle X_{1},\dots,X_{n}\rangle$
that multiplies a degree $k$ monomial by $k$. Let us also denote
by 
\[
\Sigma:\mathscr{A}_{0}\to\mathscr{A}_{0}
\]
 the inverse of $\mathscr{N}$ pre-composed by the projection $\Pi$
onto $\mathscr{A}_{0}$, given by 
\[
\Pi:P\mapsto P-P(0,0,\dots,0).
\]

If $g\in\mathscr{A}$ we write $\mathscr{D}_{j}g$ for the $j$-th
cyclic derivative of $g$ \cite{dvv:cyclomorphy}. For a monomial
$q$, $\mathscr{D}_{j}q=\sum_{q=AX_{j}B}BA$. 

Note that $\mathscr{D}_{j}g=\mathscr{D}_{j}\Pi g$. We'll denote by
$\mathscr{D}g$ the cyclic gradient (viewed as a vector):
\[
\mathscr{D}g=(\mathscr{D}_{j}g)_{j=1}^{n}.
\]
If $f=(f_{1},\dots,f_{n})$ with $f_{j}\in\mathscr{A}$ then we'll
write $\mathscr{J}f\in M_{n\times n}(\mathscr{A}\otimes\mathscr{A}^{op})$
for the matrix given by
\[
\mathscr{J}f=(\partial_{j}f_{i})_{i,j=1}^{n},
\]
where $\partial_{j}$ is the $j$-th free difference quotient \cite{dvv:entropy5}
defined as the derivation from $\mathscr{A}$ to $\mathscr{A}\otimes\mathscr{A}^{op}$
satisfying $\partial_{j}X_{i}=\delta_{j=i}$, that is if $q$ is a monomial $\partial_j q=\sum_{q=AX_jB} A\otimes B$.

\subsection{Notation: $Tr$, $\mathscr{J}^{*}$ and $\#$.}

We will often assume that $\tau$ is a positive trace and also that
$1\otimes1\in\operatorname{dom}\partial_{j}^{*}$, when $\partial_{j}:L^{2}(\mathscr{A},\tau)\to L^{2}(\mathscr{A},\tau)\bar{\otimes}L^{2}(\mathscr{A},\tau)$
is viewed as a densely defined operator. Under this assumption, $\mathscr{A}\otimes\mathscr{A}^{op}$
belongs to the domain of %
$\partial_{j}^{*}$ \cite{dvv:entropy5}.  We set for $q=\sum a_i\otimes b_i \in \mathscr{A}\otimes\mathscr{A}^{op}$
and $g\in \mathscr{A}$,
$$q\# g=\sum_i a_i g b_i\,.$$ %

For $q\in M_{n\times n}(\mathscr{A}\otimes\mathscr{A}^{op})$, and
$g=(g_{j})_{j=1}^{n}\in\mathscr{A}^{n}$, we'll write
\begin{eqnarray*}
Tr(q) & = & \sum_{i}q_{ii}\in\mathscr{A}\otimes\mathscr{A}^{op},\\ %
\mathscr{J}^{*}q & = & \left(\sum_{i}\partial_{i}^{*}(q_{ji})\right)_{j=1}^{n}\in L^{2}(\mathscr{A},\tau)^{n},\\
q\# g & = & \left(\sum_{i}q_{ji}\# g_{i}\right)_{j=1}^{n}.
\end{eqnarray*}
We will also denote by $\#$ multiplication in $M_{n\times n}(\mathscr{A}\otimes\mathscr{A}^{op})$ and write $g\# h=\sum g_{i}h_{i}$ if $g,h\in\mathscr{A}^{n}$.

\subsection{Non-commutative notions of transport.}

Let $(M,\phi)$ be a von Neumann algebra with trace $\phi$, and let
$X_{1},\dots,X_{n}\in M$ be self-adjoint elements generating $M$;
thus $M$ can be regarded as a completion of the algebra $\mathscr{A}$,
and the trace $\phi$ induces on $\mathscr{A}$ a trace $\tau_{X}$
(called the non-commutative law of $X_{1},\dots,X_{n}$). We write
$M\cong W^{*}(\tau_{X})$. Let $Y_{1},\dots,Y_{n}\in N$ be self-adjoint
elements generating another von Neumann algebra $N$ with trace $\psi$;
we denote by $\tau_{Y}$ the corresponding linear functional on $\mathscr{A}$
so that $N\cong W^{*}(\tau_{Y})$. 
\begin{defn}
By \emph{transport }from $\tau_{X}$ to $\tau_{Y}$ we mean an $n$-tuple
of self-adjoint elements $\hat{Y}_{1},\dots,\hat{Y}_{n}\in M$ having
the same law as $Y_{1},\dots,Y_{n}$. We call such transport\emph{
monotone} if $\hat{Y}=(\hat{Y}_{j})_{j}$ belongs to the $L^{2}$-closure
of the set
\[
\{\mathscr{D}g:g\in\operatorname{Alg}(X_{1},\dots,X_{n})\textrm{ and }\mathscr{J}\mathscr{D}g\geq0\}.
\] (as part of the definition of monotone transport, we are making the assumption that the variables $X_1,\dots,X_n$ are algebraically free, and so $\mathscr{J}\mathscr{D}g$ is well-defined for any $g$ in the algebra these variables generate). %
\end{defn}
When $M$ is abelian, an element of $M$ is an (essentially bounded)
function of $X_{1},\dots,X_{n}$ (and $n=1$.) %
So in the case that $M$ is abelian,
our definition reduces to the statement that the random variables
$\hat{Y}_{1},\dots,\hat{Y}_{n}$ are expressed as (bounded measurable)
functions $\hat{Y}_{j}=f_{j}(X_{1},\dots,X_{n})$. Since the law of
$Y$ is the same as the law of $\hat{Y}$, the push-forward of the
measure $\tau_{X}$ via the function $f=(f_{j})_{j}$ is exactly the
measure $\tau_{Y}$. So our notion of transport coincides with the
classical one in the commutative case.

Our definition of monotone transport is analogous to the classical
case where one requires that $\hat{Y}$ belongs to the closure of
the space of gradients of convex functions. %
However, note that a
different notion of derivative is being used. Nonetheless, when $n=1$
our requirement reduces to asking that $\hat{Y}=f(X)$ for some $f:\mathbb{R}\to\mathbb{R}$
so that $f=g'$ for some $g$ and $\frac{f(x)-f(y)}{x-y}\geq0$, i.e.,
$f$ is monotone and $g$ convex. %

Note that if a transport from $\tau_{X}$ to $\tau_{Y}$ exists, then
there exists a trace-preserving embedding $N=W^{*}(Y)$ into $M=W^{*}(X)$
given by $Y_{j}\mapsto\hat{Y}_{j}$.

\subsection{The norm $\Vert\cdot\Vert_{A\otimes_{\pi}A}$.}

For $F=(F_j)_{j=1}^n\in (\mathscr{A}^{(A)})^n$, we will use the notation $\Vert F\Vert_A=\sup_j \Vert F_j \Vert_A$. 

Let us denote by $\Vert\cdot\Vert_{A\otimes_{\pi}A}$ the projective
tensor product norm on $\mathscr{A}^{(A)}\otimes(\mathscr{A}^{(A)})^{op}$. %
In other words
$$\|\sum a_i\otimes b_i\|_{A\otimes_{\pi}A}=\sup_{PI}\|PI(\sum a_i\otimes b_i)\|$$
where the supremum is taken over all maps $PI$  valued in a Banach algebra so that %
$PI(1\otimes b)$ and $PI(a\otimes 1)$ commute and have norm bounded by the norm of $b$ and $a$ respectively.
In particular, taking $PI$ to be given by the right and left multiplication respectively, we see that 
for all $q\in \mathscr A\otimes \mathscr{A}^{op}$ and $g\in \mathscr{A}$, we have
$$\|q\# g\|_A\le \Vert q\Vert_{A\otimes_{\pi}A}\|g\|_A\,.$$ %
$\Vert\cdot\Vert_{A\otimes_{\pi}A}$ also extends to $(\mathscr A\otimes \mathscr A^{op})^n$ by putting for $F=(F_1,\ldots,F_n)\in(\mathscr A\otimes \mathscr A^{op})^n$ %
$$\Vert F\Vert_{A\otimes_{\pi}A}=\max_{1\le i\le n}\Vert F_i\Vert_{A\otimes_{\pi}A}\,.$$ %
We will also denote by the same symbol the norm on $M_{n\times n}(\mathscr{A}\otimes\mathscr{A}^{op})$
(or its completion) given by identifying that space with the Banach space of (left) multiplication operators on $(\mathscr{A}\otimes\mathscr{A}^{op})^n$.  %
Explicitly, this norm is given by $$\Vert X_{ij} \Vert_{A\otimes_\pi  A} = \max_i \sum_j \Vert X_{ij} \Vert_{A\otimes_\pi A}.$$

In consequence of its definition, we see that the maps
\begin{eqnarray*}
\#:M_{n\times n}(\mathscr{A}\otimes\mathscr{A}^{op})\times M_{n\times n}(\mathscr{A}\otimes\mathscr{A}^{op}) & \to & M_{n\times n}(\mathscr{A}\otimes\mathscr{A}^{op})\\
M_{n\times n}(\mathscr{A}\otimes\mathscr{A}^{op})\times\mathscr{A}^{n} & \to & \mathscr{A}^{n}
\end{eqnarray*}
are contractive %
 for the norm $\Vert\cdot\Vert_{A\otimes_{\pi}A}$ on $M_{n\times n}(\mathscr{A}\otimes\mathscr{A}^{op})$
and the norm $\Vert\cdot\Vert_{A}$ on $\mathscr{A}^{n}$.

\subsection{Cyclically symmetric functions.}
\begin{defn}\label{defS}
We say that a polynomial $q$ in $\mathscr{A}$ is \emph{cyclically
symmetric }if for any $i_{1},\dots,i_{k}$, the coefficient of the
monomial $X_{i_{1}}\cdots X_{i_{k}}$ in $q$ is the same as the coefficient
of the monomial $X_{i_{k}}X_{i_{1}}\cdots X_{i_{k-1}}$. 
\end{defn}
Let $\mathscr{S}$ denote the operator on $\mathscr{A}_{0}$ defined
on monomials by
\[
\mathscr{S}X_{i_{1}}\cdots X_{i_{p}}=\frac{1}{p}\sum_{r=1}^{p}X_{i_{r+1}}\cdots X_{i_{p}}X_{i_{1}}\cdots X_{i_{r}}.
\]
Clearly, $\mathscr{S}$ sends $\mathscr{A}_{0}$ to the space of cyclically
symmetric operators. Note also that $\mathscr{S}g=\Sigma \mathscr{D}g\#X=\Sigma \sum_i \mathscr{D}_i g X_i$. %
Moreover, one clearly has that $\Vert\mathscr{S}g\Vert_{A}\leq\Vert g\Vert_{A}$.
Finally, note that $\mathscr{D}(\mathscr{S}g)=\mathscr{D}g$.

\subsection{Implicit function theorem.}%

We next establish a kind of implicit function theorem for elements
of the space $(\mathscr{A}^{(A)})^n$. For later purpose note that if $F\in \mathscr{A}^{(A)}$, 
$F$ is infinitely differentiable in $ \mathscr{A}^{(A')}$ for any $A'<A$.
These results are most likely folklore,
but we were not able to find a precise reference. The proofs follow
very closely the classical proof of the implicit function theorem
using a contraction mapping principle.  
We have denoted by $\mathscr{A}$ the algebra $\mathbb{C}\langle X_{1},\dots,X_{n}\rangle $
of non-commutative polynomials in $n$ variables. Let us write $\mathscr{A}*\mathscr{A}$
for the algebra of non-commutative polynomials in $2n$ variables,
$\mathbb{C}\langle X_{1},\dots,X_{n},Y_{1},\dots,Y_{n}\rangle$. Also, remember that variables and maps are all seen as %
converging power series.
\begin{thm}
[Implicit Function Theorem]\label{theoimplicit}

Assume that $F=(F_j)_{j=1}^n$ where 
$$F_j(X_{1},\dots,X_{n},Y_{1},\dots,Y_{n})\in(\mathscr{A}*\mathscr{A})^{(A)}$$ %
is a non-commutative power series in $2n$ variables, and assume that
\[
F(0,\dots,0,0,\dots,0)=0.
\]
Let $Q=\mathscr{J}_{2}F(0,\dots,0)$ (here $\mathscr{J}_{2}$ refers
to differentiation in the variables $Y_{1},\dots,Y_{n}$, with $X_{j}$'s
held fixed), and assume that $Q$ is invertible. Then there exists
an $a_0\in (0,A)$  and a finite constant $C$ so that for all $a\in (0,a_0)$, any  $U\in\mathscr{A}^{(A)}$ satisfying $\Vert U\Vert_{A}<a$,
there exists a unique $V\in(\mathscr{A}^{(A)})^n$ satisfying $\Vert V\Vert_{A}< Ca$ %
which solves the equation
\[
F(U,V)=0.
\]
Moreover, there exists a function $f\in (\mathscr{A}^{(a)})^n$  so that $V=f(U)$. %
\end{thm}
\begin{proof}
Put $X=(X_1,\ldots,X_n)$ and $Y=(Y_1,\ldots,Y_n)$ and denote
\[
G(X,Y)=Y-Q^{-1}\#F(X,Y).
\]
Then, since $Q$ is invertible,
\[
G(U,V)=V\iff F(U,V)=0\,.
\]
Moreover, since $F(0,0)=0$, we have 
$$G(X,Y)=R(X)+ H(X,Y)$$
with $R(X)=-Q^{-1}\# \mathscr{J}_{1}F(0,\dots,0)\# X$ %
 and $H(X,Y)=G(X,Y)- R(X)$ a non-commutative power series in $\left[(\mathscr{A}*\mathscr{A})^{(A)}\right]^n$ %
 whose decomposition 
in monomials contains only monomials of degree two or greater. Therefore, for any $X,X'$ so that $\|X\|_A, \|X'\|_A\le a$
and any $Y,Y'$ so that $\|Y\|_A\le b,\|Y'\|_A\le b$, with $a,b\le A'< A$, we have 
\begin{eqnarray}
\|H(X,Y)\|_{A}&\le& C(A') (a+b)^2\label{b1}\\
\|H(X,Y)-H(X',Y')\|_A&\le& C(A') (a+b) (\|X-X'\|_A+\|Y-Y'\|_A)\label{b2}\end{eqnarray}
with a  bounded constant $C(A')$ increasing in $A'<A$. 

Choose now $0<b_0 <A'$ 
so that $C(A') b_0^2 < (1/2) b_0$, and then choose $0<a_0 <b_0/3$ so that 
 $C(A') a_0 <1/2$, $\|Q^{-1}\# \mathscr{J}_{1}F(0,\dots,0)\|_{A\otimes_\pi A} a_0 + C(A') a_0^2 <(1/2) b_0 $. These choices imply that
$C(A')(a_0^2 + 2a_0 b_0 ) < (1/2) b_0$.   

We  assume that  $U$ is given and satisfies $\|U\|_A\le a < a_0 $ and construct a solution by putting $V_{k+1}=G(U,V_{k})$ and $V_0=0$. 

We claim that $\Vert V_{k+1}\Vert_A < b_0$ for all $k$.  Indeed, $\|V_0\|_A=0\le b_0$ and 
we have by  \eqref{b1} and our choices of $a_0, b_0$ that  if $\|V_k\|_A\le b_0$, 
$$\|V_{k+1}\|_A\le \|R(U)\|_A + \|H(U, V_{k})\|_A\le  \|Q^{-1}\# \mathscr{J}_{1}F(0,\dots,0)\|_{A\otimes_\pi A} a + C(A')(a+b_0)^2<b_0.$$%
Now by \eqref{b2}, we get for all $k\ge 1$,
$$\|V_{k+1}-V_k\|_{A}\le C(A') (a+b)\|V_{k}-V_{k-1}\|_A$$
and hence $V_k$ is a Cauchy sequence since our choice implies that $C(A')(a+b)\le C(A')(a_0+b_0)<1$. Therefore it converges whenever $a\le a_0$. The limit $V_*$ satisfies 
 $F(U,V_*)=0$.
Furthermore, $\|V_*\|_A$ is bounded above by $b_0$. This proves the stated existence.

This solution is unique: indeed, \eqref{b2}  guarantees that any two solutions $V_*$ and $V_*'$ would satisfy
$$\|V_*-V_*'\|_A\le C(A') (a+b)\|V_*-V_*'\|_A$$
implying that $V_*=V_*'$, since $C(A')(a+b)<1$.

Finally, by induction
we see that for all $k$, $V_k=f_k(U)$  where $f_k(X)=G(X, f_{k-1}(X))$ and $f_0(X)=0$. As $G\in \left[(\mathscr{A}*\mathscr{A})^{(A)}\right]^n$, $f_k$ belongs to $(\mathscr{A}^{(A)})^n$
for all $k$ so that $\|f_{k-1}\|_A\le A$. Moreover,  we see as above that 
for $a<a_0$, for all $k\ge 0$,
$$\|f_k\|_a\le b_0<A$$
so that 
\begin{eqnarray*}
\|f_k-f_{k-1}\|_{a} &=&\|G(X,f_{k-1}(X))-G(X,f_{k-2}(X))\|_{a}\\
&=& \|H(X,f_{k-1}(X))-H(X,f_{k-2}(X))\|_{a}\le C(A')(a+b)\|f_{k-1}-f_{k-2}\|_a\end{eqnarray*}
and (because $C(A')(a+b)<1$) we conclude that $f_k$ converges to a limit $f$ in $(\mathscr{A}^{(a)})^n$.  Hence, $V=f(U)$ 
with $f\in (\mathscr{A}^{(a)})^n$. 
\end{proof}

\subsection{Inverses to non-commutative power series.}%

The following corollary shows that an ``absolutely convergent non-commutative
power series'' has an ``absolutely convergent inverse'':
\begin{cor}\label{corimplicit}
Let $A'<A<B$  and consider the equation $Y=X+f(X)$ with $f\in(\mathscr{A}^{(B)})^n$ and $\|Y\|_{A'}\le A$. 
Then there exists a constant $C>0$, depending only on $A, A'$ and $B$, so that whenever $\Vert f\Vert_B<C$, then there exists
$G \in(\mathscr{A}^{(A)})^n$
so that $X=G(Y)$. \end{cor}
\begin{proof}  We repeat the arguments of the previous proof. We define a sequence of maps
$$G_k= \operatorname{id}-f\circ G_{k-1}$$
with $G_0=\operatorname{id}$. By definition $\|G_0\|_{A}\le A$
and for all $k$ so that $\max_{\ell\le k} \|G_\ell\|_{A}\le B'<B$, there exists a finite constant $C(B')$ depending only on $B', B$ and $\|f\|_B$ (for example, $C(B')\le \|f\|_B  \max_{k\ge 0} k (B' )^{k-1} B^{-k}$),  so that
$$\|G_{k+1}-G_k\|_{A} \le  C(B')\|G_k-G_{k-1}\|_{A}\,.$$ 
Since  $\|G_1-G_0\|_{A}\le \|f\|_{A}<C$, we may iterate to deduce that that 
$$\|G_{k+1}-G_k\|_{A} \le (C(B'))^k \|f\|_{A}\,,$$
which implies that if $C$ is small enough so that $C(B')<1$
$$\|G_{k+1}-G_0\|_{A} \le \frac{1}{1- C(B')} \|f\|_{A} \le  \frac{C}{1- C(B')}\,.$$
Hence, since $\|G_0\|_A\le A$, if $C$ is small enough so that we can find $B'\in (A,B)$ so that $$ A+ \frac{C}{1- C(B')} \le B'$$
we are guaranteed that for all $k$, $\|G_k\|_{A}\le B'$. 
We hereafter choose $C$ so that this is satisfied and conclude that $G_k$ converges in $(\mathscr{A}^{(A)})^n$. 
Its limit  $G$ satisfies  $\|G\|_A\le B'$ and 
$$G= \textrm{id}- f\circ G.$$
Moreover since  $\|Y\|_{A'}\le A$ and $\|X\|_{A'}\le \|X\|_A\le A$,
we deduce from
$$X+f(X)=Y=G(Y)+f(G(Y))$$
that

$$\|X-G(Y)\|_{A'}\le \|f(X)-f(G(Y))\|_{A'}\le  C(B')\|X-G(Y)\|_{A'}$$
which ensures that $X=G(Y)$ since we assumed $ C(B')<1$.
\end{proof}

\section{Construction of the free monotone transport map.\label{sec:ConstructionOff}}

\subsection{Outline of the proof of the main result.}

To prove existence of the transport map, we start with an $n$-tuple
of free semicircular elements $X_{1},\dots,X_{n}$. We also fix
 $V$ of the form $V=\frac{1}{2}\sum X_{j}^{2}+W$ where
$W$ is a ``non-commutative analytic function'' of $X_{1},\dots,X_{n}$
which is sufficiently small in $\Vert\cdot\Vert_{A}$.

We then seek to
find some elements $Y_{1},\dots,Y_{n}$ in the von Neumann algebra
$M$ generated by $X_{1},\dots,X_{n}$ so that $Y_{1},\dots,Y_{n}$
have law $\tau_{V}$, which is the unique log-concave free Gibbs law
satisfying the Schwinger-Dyson equation with potential $V$, i.e.
\begin{equation}
\partial_{Y_{j}}^{*}(1\otimes1)=Y_{j}+\mathscr{D}_{Y_{j}}(W(Y_{1},\dots,Y_{n}))\label{eq:mainForY}
\end{equation}
 
To construct $Y=(Y_{1},\dots,Y_{n})$ we seek $f=(f_{1},\dots,f_{n})$
so that $Y=X+f$ (i.e., $Y_{j}=X_{j}+f_{j}$). Assuming we could find
such $Y$ with $f_{j}$ analytic and of small norm, the Schwinger-Dyson
equation \eqref{eq:mainForY} for $Y$ becomes an equation on $f$, of the form
\begin{equation*}
\mathscr{J}^{*}\left(\frac{1}{1+\mathscr{J}f}\right)=X+f+(\mathscr{D}W)(X+f).
\end{equation*}
(Here $\mathscr{J}^{*}$ is the adjoint to $\mathscr{J}$). Since
for the semicircle law, $\mathscr{J}^{*}(1)=X$, this equation further
simplifies as
\[
\mathscr{J}^{*}\left(\frac{\mathscr{J}f}{1+\mathscr{J}f}\right)+f+(\mathscr{D}W)(X+f)=0.
\]
One is tempted to solve this equation by using the contraction mapping
principle; however, one faces the obstacle of bounding the map $f\mapsto\mathscr{J}^{*}(\mathscr{J}f/(1+\mathscr{J}f))$.
Fortunately, under the assumption that $f=\mathscr{D}g$ and further
assumptions on invertibility of $1+\mathscr{J}f$, one is able to
rewrite the left hand side of the equation as the cyclic gradient
of a certain expression in $g$. This gives us an equation for $g$.
It turns out that then one can use a contraction mapping argument
to prove the existence of a solution of this equation for $g$; this
is the content of Proposition \ref{prop:gExists}. We then retrace
our steps, going from $g$ to $f=\mathscr{D}g$ and then to $Y=X+f$,
proving the main result of this section, Theorem \ref{thm:fExists}.

To avoid confusion, we will use the following convention: all differential
operators ($\mathscr{D}$, $\mathscr{J}$, $\partial$) which either
have no indices, or have a numeric index, refer to differentiation
with respect to $X_{1},\dots,X_{n}$. Operators that involve differentiation
with respect to $Y_{1},\dots,Y_{n}$ will be labeled as $\mathscr{D}_{Y_{j}}$,
$\partial_{Y_{j}}$, etc.

\subsection{A change of variables formula.}
\begin{lem}\label{change}
Put $M=W^{*}(\mathscr{A},\tau)$. Assume that $Y$ is such that $\mathscr{J}Y=(\partial_{X_{j}}Y_{i})_{ij}\in M_{n}(M\bar{\otimes}M^{op})$
is bounded and invertible. %
Assume further that $1\otimes1$ belongs
to the domain of $\partial_{j}^{*}$ for all $j$.

(i) Define
\[
\hat{\partial}_{j}(Z)=\sum_{i=1}^{n}\partial_{X_{i}}(Z)\# ((\mathscr{J}Y)^{-1})_{ij} %
\]
where $\#$ denotes the multiplication
\[
(a\otimes b)\#(A\otimes B)=aA\otimes Bb.
\]
 Then $\hat{\partial}_{j}=\partial_{Y_{j}}$ and $\partial_{Y_{j}}X_{i}=((\mathscr{J}Y)^{-1})_{ij}$. %
\\

(ii) $\partial_{Y_{j}}^{*}(1\otimes1)=\sum_\ell\partial_{X_\ell}^* ((\mathscr{J}Y)^{-1})_{\ell j}^*$. %

(iii) Assume in addition that $Y_{j}=\mathscr{D}_{j}G$ for some $G$
in the completion of $\mathscr{A}$ with respect to $\Vert\cdot\Vert_{A}$.
Assume that $G=G^{*}$. Let $(a\otimes b)^{\dagger}=b^{*}\otimes a^{*}$. Then $\mathscr{(J}Y)_{ij}^{\dagger}=(\mathscr{J}Y)_{ij}$ %
and $(\mathscr{J}Y)^{-1}$ is self-adjoint in $M_{n}(M\bar{\otimes}M^{op})$. \end{lem}
\begin{proof}
Let $Q=\mathscr{J}Y$. (i) We verify that $\hat{\partial}_{j}Y_{k}=\sum_{i}\partial_{X_i}  %
Y_k\# (Q^{-1})_{ij}=\sum_i Q_{ki}\# (Q^{-1})_{ij}= 1_{k=j} 1\otimes 1$ %

To see (ii), we compute
\begin{eqnarray*}
\langle\partial_{Y_{j}}^{*}(1\otimes1),X_{i_{1}}\dots X_{i_{p}}\rangle & = & \langle1\otimes1,\partial_{Y_{j}}(X_{i_{1}}\dots X_{i_{p}})\rangle\\
 & = & \sum_{\ell  }\langle1\otimes1, \partial_{X_\ell} (X_{i_1}\cdots X_{i_p})\# (Q^{-1})_{\ell j}\rangle\\
 &=&\sum_\ell \langle (Q^{-1})_{\ell j}^*, \partial_{X_\ell} (X_{i_1}\cdots X_{i_p})\rangle\\
 &=&\langle \sum_\ell\partial_{X_\ell}^* (Q^{-1})_{\ell j}^*,  X_{i_1}\cdots X_{i_p}\rangle
 \end{eqnarray*}

To verify (iii), we first claim that if $Y=\mathscr{D}G$ for some
$G$ (not necessarily self-adjoint), then 
\begin{equation} \label{eq:signaVSadj}
(\mathscr{J}Y)_{ij}=\sigma[(\mathscr{J}Y)_{ji}],
\end{equation}
where $\sigma(a\otimes b)=b\otimes a$. Indeed, let $G=X_{i_{1}}\cdots X_{i_{k}}$;
in this case
\[
\partial_{i}Y_{j}=\sum_{G=AX_{i}BX_{j}C}CA\otimes B+\sum_{G=AX_{j}BX_{i}C}B\otimes CA. \]
Reversing the role of $i$ and $j$ amounts to switching the two sums,
i.e., applying $\sigma$. 

We next consider $G$ of the form
\[
G=X_{i_{1}}\cdots X_{i_{k}}+X_{i_{k}}\cdots X_{i_{1}}=P+P^{*},
\]
where $P=X_{i_{1}}\dots X_{i_{k}}.$ In that case
\[
Y_{j}=\sum_{P=AX_{j}B}BA+A^{*}B^{*}
\]
and therefore we have
\[
(\mathscr{J}Y)_{ji}=\partial_{i}Y_{j}=\sum_{P=AX_{j}B}\sum_{BA=RX_{i}S}R\otimes S+S^{*}\otimes R^{*}.%
\]
Thus, if we denote by $*$ the involution $(a\otimes b)^{*}=a^{*}\otimes b^{*}$, we deduce that
\begin{equation}\label{kl}
(\mathscr{J}Y)_{ij}^{*}=\sigma((\mathscr{J}Y)_{ij})=(\mathscr{J}Y)_{ji}
\end{equation}
proving that $\mathscr{J}Y$ is self-adjoint. Moreover we find that
\begin{equation}\label{kl2}
((\mathscr{J}Y)_{ij})^{\dagger}=\sigma((\mathscr{J}Y)_{ij}^{*})=\sigma((\mathscr{J}Y)_{ji})=(\mathscr{J}Y)_{ij}.
\end{equation}
where the last two equalities come from \eqref{kl}.%

\end{proof}

A direct consequence of Lemma \ref{change} (ii) and (iii), is the following. %
 
\begin{cor}
Assume that $g\in\mathscr{A}^{(A)}$, $1\otimes1\in\operatorname{dom}\partial_{j}^{*}$
for all $j$, and put $G=\frac{1}{2}\sum X_{j}^{2}+g$. Let $f_j =\mathscr{D}_j g$ and $Y_{j}=X_{j}+f_{j}$
so that $Y=\mathscr{D}G$. Assume that $1+\mathscr{J}f$ is invertible.
Then Equation (\ref{eq:mainForY}) is equivalent to the equation
\begin{equation}
\mathscr{J}^{*}\left(\frac{1}{1+\mathscr{J}f}\right)=X+f+(\mathscr{D}W)(X+f).\label{eq:Main-1}
\end{equation}

\end{cor}

\subsection{An equivalent form of Equation (\ref{eq:Main-1}).}
\begin{lem}
\label{lemma:AfterSharp}Let $\tau$ be the semicircle law. Assume
that the map $\xi\mapsto(\mathscr{J}f)\#\xi+\xi$ is invertible on
$(\mathscr{A}^{(A)})^{n}$, and that $f=\mathscr{D}g$ for some $g=g^*$. %
Let
\[
K=-\mathscr{J}^{*}\circ\mathscr{J}-\operatorname{id}.
\]
Then equation (\ref{eq:Main-1}) is equivalent to
\begin{equation}
K(f)=\mathscr{D}(W(X+f))+\left[(\mathscr{J}f)\#f+(\mathscr{J}f)\#\mathscr{J}^{*}\left(\frac{\mathscr{J}f}{1+\mathscr{J}f}\right)-\mathscr{J}^{*}\left(\frac{(\mathscr{J}f)^{2}}{1+\mathscr{J}f}\right)\right].\label{eq:AfterSharp}
\end{equation}
\end{lem}
\begin{proof}
Using the formula $\frac{1}{1+x}=1-\frac{x}{1+x}$ and the fact that
$\mathscr{J}^{*}(1\otimes1)=X$, we see that (\ref{eq:Main-1}) is
equivalent to
\[
\mathscr{J}^{*}\left(\frac{\mathscr{J}f}{1+\mathscr{J}f}\right)+f+(\mathscr{D}W)(X+f)=0. %
\]
We will now apply $(1+\mathscr{J}f)\#$ to both sides of the equation
(this map is, by assumption, invertible, so the resulting equation
is equivalent to the one in the previous line):
\begin{multline*}
\mathscr{J}^{*}\left(\frac{\mathscr{J}f}{1+\mathscr{J}f}\right)+f+(\mathscr{D}W)(X+f)+(\mathscr{J}f)\#\mathscr{J}^{*}\left(\frac{\mathscr{J}f}{1+\mathscr{J}f}\right)\\
+(\mathscr{J}f)\#f+(\mathscr{J}f)\#(\mathscr{D}W)(X+f)=0.
\end{multline*}
Using that $\frac{x}{1+x}=x-\frac{x^{2}}{1+x}$, we get:
\begin{multline*}
\mathscr{J}^{*}(\mathscr{J}f)-\mathscr{J}^{*}\left(\frac{(\mathscr{J}f)^{2}}{1+\mathscr{J}f}\right)+f+(\mathscr{D}W)(X+f)\\
+(\mathscr{J}f)\#\mathscr{J}^{*}\left(\frac{\mathscr{J}f}{1+\mathscr{J}f}\right)+(\mathscr{J}f)\#f+(\mathscr{J}f)\#(\mathscr{D}W)(X+f)=0.
\end{multline*}
Thus we have:
\begin{multline*}
K(f)=(\mathscr{D}W)(X+f)+(\mathscr{J}f)\#(\mathscr{D}W)(X+f)\\
+\left[(\mathscr{J}f)\#f+(\mathscr{J}f)\#\mathscr{J}^{*}\left(\frac{\mathscr{J}f}{1+\mathscr{J}f}\right)-\mathscr{J}^{*}\left(\frac{(\mathscr{J}f)^{2}}{1+\mathscr{J}f}\right)\right].
\end{multline*}
We now note that, because of cyclic symmetry and the fact that $f = \mathscr{D} g$ %
 with $g=g^*$, we may use \eqref{kl2}  to deduce:
\begin{equation*}
\mathscr{D}(W(X+f)) %
=(\mathscr{D}W)(X+f)+(\mathscr{J}f)\#(\mathscr{D}W)(X+f) %
\end{equation*} %
 Hence (\ref{eq:Main-1}) is equivalent
to
\[
K(f)=\mathscr{D}(W(X+f))+\left[(\mathscr{J}f)\#f+(\mathscr{J}f)\#\mathscr{J}^{*}\left(\frac{\mathscr{J}f}{1+\mathscr{J}f}\right)-\mathscr{J}^{*}\left(\frac{(\mathscr{J}f)^{2}}{1+\mathscr{J}f}\right)\right],
\]
as claimed.
\end{proof}

\subsection{Some identities involving $\mathscr{J}$ and $\mathscr{D}$.}
We start with the following identity.
\begin{lem}
\label{lem:DvsPartialStar}Put $M=W^*(\mathcal A,\tau)$. Assume that $1\otimes1$ belongs to the
domain of $\partial_{j}^{*}$ for all $j$ and that $\tau$ is a faithful trace, i.e., $\tau(P^*P)=0$ iff $P=0$. Let $g\in \mathscr{A}^{(A)}$ %
 and let
$f=\mathscr{D}g$. Then for any $m\geq-1$ we have:
\begin{multline}
\frac{1}{m+2}\mathscr{D}\left[(1\otimes\tau+\tau\otimes1)\left\{ \sum_{i=1}^{n}\left[(\mathscr{J}f)^{m+2}\right]_{ii}\right\} \right]\\
=-\mathscr{J}^{*}\left((\mathscr{J}f)^{m+2}\right)+\mathscr{J}f\#\mathscr{J}^{*}\left((\mathscr{J}f)^{m+1}\right)\,.\label{theeqr}
\end{multline} 
\end{lem}
\begin{proof}
Let us take  $f=\mathscr{D}g$. We shall prove that \eqref{theeqr} is true by showing that it holds weakly. 
We therefore let   $P$ be a test function in $(\mathcal A^{(A)})^n$.  Then, with the notations of the proof of Lemma
\ref{change}, see \eqref{eq:signaVSadj},  and with $\#$  the multiplication in $(\mathcal A^{(A)})^n$, we have 

\begin{eqnarray*}
\Lambda&:=&\tau\left(\left[-\mathscr{J}^{*}\left((\mathscr{J}f)^{m+2}\right)+\mathscr{J}f\#\mathscr{J}^{*}\left((\mathscr{J}f)^{m+1}\right)\right]\#P\right)\\
&=&-\tau\left(\mathscr{J}^{*}\left((\mathscr{J}f)^{m+2}\right)\#P\right)+\tau\left(\sum_{ij}(\mathscr{J}f)_{ij}\#\left(\mathscr{J}^{*}\left((\mathscr{J}f)^{m+1}\right)\right)_{j}\#P_{i}\right)\\
&=&-\tau\left(\mathscr{J}^{*}\left((\mathscr{J}f)^{m+2}\right)\#P\right)+\tau\left(\sum_{ij}\left(\mathscr{J}^{*}\left((\mathscr{J}f)^{m+1}\right)\right)_{j}\#\sigma(\mathscr{J}f)_{ij}\#P_{i}\right)
\end{eqnarray*}
where we just used that $\tau$ is tracial.
Since $f=\mathscr{D}g$, it follows that $\sigma (\mathscr{J}f)_{ij}=(\mathscr{J}f)_{ji}$ by  \eqref{eq:signaVSadj}.
Therefore we deduce that
\begin{eqnarray*}
\Lambda
&=&-\tau\left(\mathscr{J}^{*}\left((\mathscr{J}f)^{m+2}\right)\#P\right)+\tau\left(\sum_{ij}\left(\mathscr{J}^{*}\left((\mathscr{J}f)^{m+1}\right)\right)_{j}\#(\mathscr{J}f)_{ji}\#P_{i}\right)\\
&=&
-\sum_{ij} \tau\otimes\tau
\left((\mathscr{J}f)^{m+2}_{ij} \#\sigma \left[\mathscr{J}P\right]_{ij} \right)
+\sum_{ij} \tau\otimes\tau \left(\left(\mathscr{J}f\right)^{m+1}_{ij} \#
\sigma \left[\mathscr{J}\left((\mathscr{J}f)\#P\right)\right]_{ij} \right)
\end{eqnarray*}
where in the last equality  we used the definition of $\mathscr{J}^{*}$: $$
\tau\left( \mathscr{J}^*(U) \# W\right) = \sum_{ij} \tau\otimes\tau \left(   U_{ij} \# \sigma \left[ \mathscr{J} W\right]_{ij} \right).$$ 

Let us consider the term $\sigma \left[\mathscr{J}\left((\mathscr{J}f)\#P\right)\right]$.  
Using the Leibniz rule, we obtain
\begin{eqnarray*}
\left[\mathscr{J}\left((\mathscr{J}f)\#P\right)\right]_{ij}&=& \partial_j \sum_k (\partial_k f_i)\# P_k \\
&=&   \sum_k (\partial_k f_i) \# \partial_j P_k + \sum_k ( (\partial_j\otimes 1)(\partial_k f_i)) \#_2 P_k + \sum_k ((1\otimes \partial_j)(\partial_k f_i))\#_1 P_k
\end{eqnarray*}
where we use the notation $(a\otimes b\otimes c)\#_{1}\xi=a\xi b\otimes c$,
$(a\otimes b\otimes c)\#_{2}\xi=a\otimes b\xi c$. 

Thus 
\begin{eqnarray*} 
 \sigma \left[\mathscr{J}\left((\mathscr{J}f)\#P\right)\right]_{ij} &=&  
 \sum_k \sigma ((\partial_k f_i) \# \partial_j P_k) 
 \\ &&  + \sum_k \sigma[ ( (\partial_j\otimes 1)(\partial_k f_i)) \#_2 P_k] + \sum_k \sigma[((1\otimes \partial_j)(\partial_k f_i))\#_1 P_k]
 \\ &=& \sum_k  \sigma ( \mathscr{J}P)_{kj}  \# \sigma (\mathscr{J} f)_{ik}  
 \\ && + \sum_k \eta[ ( (\partial_j\otimes 1)(\partial_k f_i))] \#_1 P_k + \sum_k \eta[((1\otimes \partial_j)(\partial_k f_i))]\#_2 P_k
 \end{eqnarray*}
 where $\eta[a\otimes b \otimes c] =b\otimes c\otimes a.$
Since $f=\mathscr{D}g$, we know that $\sigma (\mathscr{J}f)_{ik}=(\mathscr{J}f)_{ki}$, so that
\begin{eqnarray*}  
 \sigma \left[\mathscr{J}\left((\mathscr{J}f)\#P\right)\right]_{ij} &=& \sum_k \sigma ( \mathscr{J}P)_{kj}  \#  (\mathscr{J} f)_{ki}  
 \\ && + \sum_k \eta[ ( (\partial_j\otimes 1)(\partial_k f_i))] \#_1 P_k + \sum_k \eta[((1\otimes \partial_j)(\partial_k f_i))]\#_2 P_k
 \end{eqnarray*}

We now apply the identity 
$$\eta\left[ (\partial_j \otimes 1)\partial_k \mathscr{D}_i g \right]=  (\partial_k \otimes 1) \partial_i \mathscr{D}_j g 
= (1\otimes \partial_i) \partial_k \mathscr{D}_j g $$ (and a similar one involving $\eta[((1\otimes \partial_j)(\partial_k f_i))]$) to deduce that
\begin{eqnarray*} 
 \sigma \left[\mathscr{J}\left((\mathscr{J}f)\#P\right)\right]_{ij} &=& \sum_k \sigma ( \mathscr{J}P)_{kj}  \# \sigma (\mathscr{J} f)_{ik}  
 \\ && + \sum_k  ( (1\otimes \partial_i)(\partial_k f_j))\#_1 P_k + \sum_k ((\partial_i\otimes 1)(\partial_k f_j))\#_2 P_k
\\ &=&  \sigma ( \mathscr{J}P)_{kj}  \# \sigma (\mathscr{J} f)_{ik} + Q^P_{ji}
\end{eqnarray*}
where
\[
Q_{ji}^P=\sum_{k}(1\otimes\partial_{i})\partial_{k}f_{j}\#_{1}P_{k}+\sum_{k}(\partial_{i}\otimes1)\partial_{k}f_{j}\#_{2}P_{k}.
\]

Thus if we set $Q^P=((Q^P)_{ij})_{ij}$ and $R=(R_{ij})$ with $R_{ij} = \sigma(\mathscr{J}P)_{ji}$, then
\begin{eqnarray}
\Lambda
&=&-\tau\otimes\tau\otimes Tr\left((\mathscr{J}f)^{m+2}\#R\right)+\tau\otimes\tau\otimes Tr\left(\left(\mathscr{J}f\right)^{m+1}\#
 R \#\left(\mathscr{J}f\right) \right)\nonumber\\
&&\qquad\qquad\qquad+\tau\otimes\tau\otimes Tr\left(\left(\mathscr{J}f\right)^{m+1}\#Q^P\right)\nonumber\\
&=&\tau\otimes\tau\otimes Tr\left(\left(\mathscr{J}f\right)^{m+1}\#Q^P\right)\nonumber\\
&=&\frac{1}{m+2}\sum_{h=0}^{m+1}\tau\otimes\tau\otimes Tr\left(\left(\mathscr{J}f\right)^{h}\#Q^P\#\left(\mathscr{J}f\right)^{m+1-h}\right).\label{eq:diffOfJs}
\end{eqnarray}
We therefore need to prove that
\begin{eqnarray}
&&\sum_{h=0}^{m+1}\tau\otimes\tau\otimes Tr\left(\left(\mathscr{J}f\right)^{h}\#Q^P\#\left(\mathscr{J}f\right)^{m+1-h}\right)\qquad\label{theeqs}\\
&&=\tau\left(P\# \mathscr{D}\left[(1\otimes\tau+\tau\otimes1)\left\{ Tr (\mathscr{J}f)^{m+2}\right\} \right]\right)\nonumber\end{eqnarray}
holds for all test function $P$.  We prove this equality by showing that both sides are equal to  the derivative of the same quantity. 
Let $X_{j}^{t}=X_{j}+tP_{j}$, and consider
\begin{eqnarray*} %%DS: changed #_1 to #_2 etc to match what we want to prove!
\frac{d}{dt}\Big|_{t=0}(\mathscr{J}f(X^{t})_{ij}) & = & \frac{d}{dt}\Big|_{t=0}\partial_{j}f_{i}(X^{t})=\sum_{k}(1\otimes\partial_{k})\partial_{j}f_{i}\#_{2}P_{k}+\sum_{k}(\partial_{k}\otimes1)\partial_{j}f_{i}\#_{1}P_{k}\\
 & &\qquad = \sum_{k}(1\otimes\partial_{j})\partial_{k}f_{i}\#_{1}P_{k}+\sum_{k}(\partial_{j}\otimes1)\partial_{k}f_{i}\#_{2}P_{k}
\end{eqnarray*}
where we have used the identity $(\partial_{j}\otimes1)\partial_{k}=(1\otimes\partial_{k})\partial_{j}$.
It follows that
\[
Q_{ij}^P=\frac{d}{dt}\Big|_{t=0}\left(\left[\mathscr{J}f(X^{t})\right]_{ij}\right).
\]
Thus, we deduce 
\[
\sum_{t=0}^{m+1}\tau\otimes\tau\otimes Tr\left(\left(\mathscr{J}f\right)^{h}\#Q^P\#\left(\mathscr{J}f\right)^{m+1-h}\right)=\frac{d}{dt}\Big|_{t=0}\tau\otimes\tau\otimes Tr\left(\left(\mathscr{J}f(X^{t})\right)^{m+2}\right).
\]

On the other hand, if $R\in M_{n\times n}(\mathscr{A}\otimes\mathscr{A}^{op})$
is arbitrary, then we claim that 
\begin{equation}
\frac{d}{dt}\Big|_{t=0}\tau\otimes\tau\otimes Tr\left(R(X^{t})\right)=\sum_{i}\tau(\mathscr{D}(1\otimes\tau+\tau\otimes1)(R_{ii})\#P).\label{eq:tauOfDtauotimes1}
\end{equation}
Indeed, this is sufficient to check (\ref{eq:tauOfDtauotimes1}) for
$R$ a matrix with a single nonzero component $A\otimes B$ in the
$i,j$-th position. For this choice of $R$,  we have 
\begin{eqnarray*}
\frac{d}{dt}\Big|_{t=0}\tau\otimes\tau\otimes Tr(R(X^{t})) & = & \frac{d}{dt}\Big|_{t=0}\delta_{i=j}\tau(A(X^{t}))\tau(B(X^{t}))\\
 & = & \delta_{i=j}\left(\frac{d}{dt}\Big|_{t=0}\tau(A(X^{t}))\tau(B)\right)+\left(\frac{d}{dt}\Big|_{t=0}\tau(A)\tau(B(X^{t}))\right)\\
 & = & \delta_{i=j}\left(\tau(\mathscr{D}A\#P\right)\tau(B)+\tau(A)\tau(\mathscr{D}B\#P))\\
 & = & \sum_{i}\tau\left(\mathscr{D}((1\otimes\tau+\tau\otimes1)(R_{ii}))\#P\right),
\end{eqnarray*}
as claimed.
We now combine (\ref{eq:diffOfJs}) and (\ref{eq:tauOfDtauotimes1}) 
(with $R=(\mathscr{J}f)^{m+2}$) to conclude that \eqref{theeqs} holds. Since this identity holds for any $P$, the Lemma follows.
\end{proof}
\begin{lem}
\label{lemma:QvsJ}Assume that $f=\mathscr{D}g$, $g=g^{*}\in\mathscr{A}^{(A)}$,
 assume that $1\otimes1\in\operatorname{dom}\partial_{j}^{*}$
for all $j$ and that $\tau$ is a faithful trace, i.e., $\tau(P^*P)=0$ iff $P=0$. Assume that $\Vert \mathscr{J}f\Vert_{A\otimes_\pi A} <1$ %
and let
\[
Q(g)=\left[(1\otimes\tau+\tau\otimes1)\left\{ \sum_{i=1}^{n}\left(\mathscr{J}f\right)_{ii}-\left(\log(1+\mathscr{J}f)\right)_{ii}\right\} \right].
\] Then
\begin{eqnarray*}
\mathscr{D}Q & = & \mathscr{J}f\#\mathscr{J}^{*}\left(\frac{\mathscr{J}f}{1+\mathscr{J}f}\right)-\mathscr{J}^{*}\left(\frac{(\mathscr{J}f)^{2}}{1+\mathscr{J}f}\right).
\end{eqnarray*}
\end{lem}
\begin{proof}
We use Lemma \ref{lem:DvsPartialStar} and compare the (converging) %
Taylor series
expansions of both sides term by term.\end{proof}
\begin{lem}\label{lem:KforVarm}
\label{lem:KasD}Assume that $\tau$ is the semicircle law. Let 
\[
K(f)=-\mathscr{J}^{*}(\mathscr{J}f)-f.
\]
Assume that $f=\mathscr{D}g$, and that $g\in\mathscr{A}$ is cyclically
symmetric. Then
\begin{eqnarray*}
K(f) & = & \mathscr{D}\left\{ (1\otimes\tau+\tau\otimes1)\left(\sum_{i}\left[\mathscr{J}f\right]_{ii}\right)-f\#X\right\} \\
 & = & \mathscr{D}\left\{ (1\otimes\tau+\tau\otimes1)\left(Tr\left[\mathscr{J}\mathscr{D}g\right]\right)-\mathscr{N}g\right\} .
\end{eqnarray*}
\end{lem}
\begin{proof}
When $m=-1$, %
 the equality in Lemma \ref{lem:DvsPartialStar}  %
becomes
\[
\mathscr{D}\left\{ (1\otimes\tau+\tau\otimes1)\left(Tr\left[\mathscr{J}f\right]\right)\right\} =-\mathscr{J}^{*}\left(\mathscr{J}f\right)+\mathscr{J}f\#\mathscr{J}^{*}\left(I\right),
\]
where $I\in M_{n\times n}(\mathscr{A}\otimes\mathscr{A}^{op})$ denotes
the identity matrix. 
Since $\tau$ is the semicircle law, $\mathscr{J}^{*}(I)=X$,
and $\mathscr{J}f\#X=\mathscr{N}f$ (here $\mathscr{N}$ is applied
entrywise to the vector $f$). 

On the other hand, since $\mathscr{D}$
reduces the degree of polynomials by $1$, $\mathscr{N}\mathscr{D}g=\mathscr{D}\mathscr{N}g-\mathscr{D}g=\mathscr{D}\mathscr{N}g-f$.
Thus we deduce %
\begin{eqnarray*}
\mathscr{D}\left\{ (1\otimes\tau+\tau\otimes1)\left(Tr\left[\mathscr{J}\mathscr{D}g\right]\right)-\mathscr{N}g\right\}  & = & -\mathscr{J}^{*}(\mathscr{J}f)+\mathscr{N}\mathscr{D}g-\mathscr{D}\mathscr{N}g\\
 & = & K(f).
\end{eqnarray*}
 Using the equality $f\#X=\mathscr{D}g\#X=\mathscr{N}g$ due to the fact that $g$ is cyclically symmetric, we get the
statement of the Lemma.\end{proof}
\begin{lem}
\label{lem:IfWeCanSolveFreeMongeAmpere}Let $\tau$ be the semicircle
law. Assume that $f=\mathscr{D}g$ for some
$g=g^*$ cyclically symmetric and that $\Vert\mathscr{J}f\Vert_{A\otimes_\pi A}<1$. %
Let
\[
Q(g)=\left[(1\otimes\tau+\tau\otimes1)Tr\left\{ \left(\mathscr{J}\mathscr{D}g\right)-\left(\log(1+\mathscr{J}\mathscr{D}g)\right)\right\} \right].
\]
Then Equation (\ref{eq:Main-1}) is equivalent to the equation
\begin{eqnarray*}
\mathscr{D}\left\{ (1\otimes\tau+\tau\otimes1)\left(Tr\left[\mathscr{J}\mathscr{D}g\right]\right)-\mathscr{N}g\right\} %
 & = & \mathscr{D}(W(X+\mathscr{D}g))+\mathscr{D}Q(g)+\mathscr{J}\mathscr{D}g\#\mathscr{D}g.
\end{eqnarray*}
\end{lem}
\begin{proof}
By Lemma \ref{lem:KforVarm}, the left-hand side is precisely $K(f)$.
By Lemma \ref{lemma:AfterSharp},
we 
obtain
\[
K(f)=\mathscr{D}(W(X+f))+\left(\mathscr{J}f\#\mathscr{J}^{*}\left(\frac{\mathscr{J}f}{1+\mathscr{J}f}\right)-\mathscr{J}^{*}\left(\frac{(\mathscr{J}f)^{2}}{1+\mathscr{J}f}\right)+\mathscr{J}f\#f\right)
\]
which, according to Lemma \ref{lemma:AfterSharp} is
equivalent to (\ref{eq:Main-1}). The hypothesis
of Lemma \ref{lemma:AfterSharp} is satisfied since the map  $T:\xi\mapsto\mathscr{J}f\#\xi$
is strictly contractive on $(\mathscr{A}^{(A')})^{n}$ because of the bound on $\mathscr{J}f$ (and so  the map $\operatorname{id}+T$
has no kernel). 
\end{proof}
We now turn to the proof of existence of a $g$ that satisfies the
equation above. We will use a fixed-point argument; thus we first
give some estimates on differential operators that will be involved.

\subsection{Technical estimates on certain differential operators.}

For the remainder of the section, we will assume that $\tau:\mathscr{A}\to\mathbb{C}$
satisfies
\begin{equation}
|\tau(q)|\leq C_{0}^{\deg q}\label{eq:boundOnTau}
\end{equation}
for any monomial $q\in\mathscr{A}$. We begin with a few technical
estimates on certain differential operators.
\begin{lem} \label{lem:EstimateQm}
Let $g_{1},\dots,g_{m}\in\mathscr{A}_{0}$. Set
\[
Q_{m}(g_{1},\dots,g_{m})=(1\otimes\tau+\tau\otimes1)\left\{ \sum_{i=1}^{n}\left[(\mathscr{J}\mathscr{D}g_{1})\cdots(\mathscr{J}\mathscr{D}g_{m})\right]_{ii}\right\} .
\]
Assume that (\ref{eq:boundOnTau}) holds and moreover that $C_{0}/A<1/2$.
Then 
\[
\Vert Q_{m}(\Sigma g_{1},\dots,\Sigma g_{m})\Vert_A \leq 2 (2A^{-2})^{m}\prod_{k=1}^{m}\Vert g_{k}\Vert_{A}. %
\]
In particular, $Q_{m}$ extends to a bounded multilinear operator
on $\mathscr{A}^{(A)}$ with values in $\mathscr{A}^{(A)}$.\end{lem}
\begin{proof}
Note that by definition, %
assuming that $g$ is a monomial, we have
\begin{eqnarray*}
[\mathscr{J}\mathscr{D}\Sigma g]_{ji} & = & \frac{1}{\deg g}\partial_{i}\mathscr{D}_{j}g\\
 & = & \frac{1}{\deg g}\sum_{g=AX_{j}B}\sum_{BA=RX_{i}Q}R\otimes Q.
\end{eqnarray*}
Thus if $g_{1},\dots,g_{m}$ are monomials, then
\begin{equation*}
[(\mathscr{J}\mathscr{D}\Sigma g_{1})\cdots(\mathscr{J}\mathscr{D}\Sigma g_{m})]_{j_{0}j_{m}}
=\frac{1}{\prod_t\deg g_{t}}  \sum_{
{j_{1},\dots,j_{m-1}} \atop 
{g_{k}=A_{k}X_{j_{k}}B_{k}}
}
\sum_{B_{k}A_{k}=R_{k}X_{j_{k+1}}Q_{k}}R_{1}\cdots R_{m}\otimes Q_{1}\cdots Q_{m}
\end{equation*}
and so, using the fact that for a given $k$ there can be at most
$\deg g_{k}$ decompositions of $g_{k}$ as $A_{k}X_{j_{k+1}}B_{k}$,
and the degree of $A_{k}$ determines such decompositions, we conclude
that each such sum has at most $\prod_t \deg g_t$ nonzero terms, so that  %
\begin{eqnarray*}
&&\Vert(1\otimes\tau)[(\mathscr{J}\mathscr{D}\Sigma g_{1})\cdots(\mathscr{J}\mathscr{D}\Sigma g_{m})]_{j_{0}j_{m}}\Vert_{A}\\
&&\quad \leq
\frac{1}{\prod_t\deg g_{t}}  \sum_{
{j_{1},\dots,j_{m-1}} \atop 
{g_{k}=A_{k}X_{j_{k}}B_{k}}
}
\sum_{B_{k}A_{k}=R_{k}X_{j_{k+1}}Q_{k}}
A^{\deg(R_{1}\cdots R_{m})}\otimes|\tau(Q_{1}\cdots Q_{m})|\\
&&\quad\leq \sum_{l_1,\dots,l_m} A^{l_{1}+\cdots+l_{m}}\cdot C_{0}^{\deg g_{1}-l_{1}-2+\cdots+\deg g_{m}-l_{m}-2}\\
&&\quad =A^{\deg g_{1}+\cdots+\deg g_{m}}A^{-2m}\sum_{l_{1}\cdots l_{m-1}}\left(\frac{C_{0}}{A}\right)^{\deg g_{1}-l_{1}-2+\cdots+\deg g_{m}-l_{m}-2}\\
&&\quad \leq \prod_{k=1}^{m}[A^{-2}\Vert g_{k}\Vert_{A}\frac{1}{1-C_{0}/A}]\leq \prod_{k=1}^{m}[2A^{-2}\Vert g_{k}\Vert_{A}]. %
\end{eqnarray*}
A similar estimate holds for $(\tau\otimes1)([(\mathscr{J}\mathscr{D}\Sigma g_{1})\cdots(\mathscr{J}\mathscr{D}\Sigma g_{m})]_{j_{0}j_{m}})$. %

Since $Q_{m}(\Sigma g_{1},\dots,\Sigma g_{m})$ is multi-linear in $g_{1},\dots,g_{m}$, %
we find that if $g_{j}=\sum\lambda(j,q)q$ is the decomposition of
$g_{j}$ in terms of monomials, then we get that
\[
Q_{m}(\Sigma g_{1},\dots,\Sigma g_{m})=\sum_{q_{1},\dots,q_{m}}\lambda(1,q_{1})\cdots\lambda(m,q_{m})Q_{m}(\Sigma q_{1},\dots,\Sigma q_{m})
\]
and so
\begin{eqnarray*}
\Vert Q_{m}(\Sigma g_{1},\dots,\Sigma g_{m})\Vert_{A} & \leq & 2(2A^{-2})^{m}\sum_{q_{1},\dots,q_{m}}\prod_{j=1}^{m}|\lambda(j,q_{j})|\Vert q_{j}\Vert_{A}\\
 & = & 2(2A^{-2})^{m}\prod_{j=1}^{m}\sum_{q_{j}}|\lambda(j,q_{j})|\Vert q_{j}\Vert_{A}\\
 & = & 2(2A^{-2})^{m}\prod_{j=1}^{m}\Vert g_{j}\Vert_{A}.
\end{eqnarray*}
 This concludes the proof for arbitrary $g_1,\dots,g_m$.
\end{proof}
\begin{lem}\label{lipQ}
For $g,f\in\mathscr{A}_{0}$, set $Q_{m}(\Sigma g)=Q_{m}(\Sigma g,\dots,\Sigma g)$. %
Assume
that (\ref{eq:boundOnTau}) holds and moreover that $C_{0}/A<1/2$.
Then
\[
\Vert Q_{m} (\Sigma g) - Q_m (\Sigma f) \Vert_A\leq 2(2A^{-2})^m \sum_{k=0}^{m-1} \Vert g\Vert_A^{k} \Vert f\Vert_A^{m-k-1} \Vert f-g\Vert_A
\]
In particular, $\Vert Q_{m}(\Sigma g)\Vert_{A}\leq2 (2A^{-2})^{m}\Vert g\Vert_{A}^{m}$.\end{lem}
\begin{proof}
We perform a telescopic expansion
$$
\Vert Q_{m}(\Sigma f)-Q_{m}(\Sigma g)\Vert_{A} = \left\Vert \sum_{k=0}^{m-1}Q_{m}(\underbrace{\Sigma g,\dots,\Sigma g}_{k},\underbrace{\Sigma f,\dots,\Sigma f}_{m-k})-Q_{m}(\underbrace{\Sigma g,\dots,\Sigma g}_{k+1},\underbrace{\Sigma f,\dots,\Sigma f}_{m-k-1})\right\Vert _{A}$$
\begin{eqnarray*}
 & \leq & \sum_{k=0}^{m-1}\Vert Q_{m}(\underbrace{\Sigma g,\dots,\Sigma g}_{k},\Sigma f-\Sigma g,\underbrace{\Sigma f,\dots,\Sigma f}_{m-k-1})\Vert_{A}\\
 & \leq & 2(2A^{-2})^{m}\sum_{k=0}^{m-1}\Vert g\Vert_{A}^{k}\Vert f\Vert_{A}^{m-k-1}\Vert f-g\Vert_{A}\,.
 \end{eqnarray*} %
\end{proof}
\begin{cor}
Assume that (\ref{eq:boundOnTau}) holds and moreover that $C_{0}/A<1/2$. %
Then the maps $g\mapsto Q_{m}(\Sigma g)$ extend
by continuity to the completion of $\mathscr{A}_{0}$ with respect
to $\Vert\cdot\Vert_{A}$. \end{cor}
\begin{lem}
\label{lemma:DefOfQ}Assume that (\ref{eq:boundOnTau}) holds and
moreover that $C_{0}/A<1/2$. Let $g\in\mathscr{A}^{(A)}_0$ %
 be such
that $\Vert g\Vert_{A}<\frac{A^{2}}{2}$, and set
\[
Q(\Sigma g)=\sum_{m\geq0}\frac{(-1)^{m}}{m+2}Q_{m+2}(\Sigma g).
\]
Then this series converges in $\Vert\cdot\Vert_{A}$. Moreover, in
the sense of analytic functional calculus on $M_{n\times n}(W^{*}(\mathscr{A}\otimes\mathscr{A}^{op},\tau\otimes\tau^{op}))$,
we have equality
\[
Q(\Sigma g)=\left[(1\otimes\tau+\tau\otimes1)\left\{ \sum_{i=1}^{n}\left(\mathscr{J}\mathscr{D}\Sigma g\right)_{ii}-\left(\log(1+\mathscr{J}\mathscr{D}\Sigma g)\right)_{ii}\right\} \right].
\]
Furthermore, the function $Q$ satisfies the (local) Lipschitz condition on $\{g:\|g\|_A < A^2/2\}$ %
\[
\Vert Q(\Sigma g)-Q(\Sigma f)\Vert_A \leq \Vert f-g \Vert_A \frac{2}{A^2}\left(\frac{1}{\left(1-\frac{2\Vert f\Vert_A}{A^2}\right)\left(1-\frac{2\Vert g\Vert_A}{A^2}\right)}-1\right),
\]
and the bound
\[
\Vert Q(\Sigma g)\Vert_A \leq \frac{\left(\frac{2\Vert g\Vert_A}{A^2}\right)^2}{1-\frac{2\Vert g\Vert_A}{A^2}}\,.
\]

\end{lem}
\begin{proof}
Let $\kappa=\frac{A^{2}}{2}$ and $\lambda=\|g\|_A$. By Lemma \ref{lem:EstimateQm}, we have that $\Vert Q_{m+2}(g)\Vert_A \le 2(\lambda/\kappa)^{(m+2)}$. Thus, as we assumed  $\lambda<\kappa$,
  the series %
defining $Q(\Sigma g)$ converges. 
 To see the claimed equality with the expression given by functional
calculus, we only need to note that $\log(1+x)=-\sum_{m\geq1}\frac{(-x)^{m}}{m}.$
Finally, since $m+2\geq 2$ in our series, we obtain that
\begin{eqnarray*}
\Vert Q(\Sigma g)-Q(\Sigma f)\Vert_{A}&\leq&\sum_{m\geq 0}\frac{1}{m+2}\Vert Q_{m+2}(\Sigma g)-Q_{m+2}(\Sigma f)\Vert_{A}\\
&\leq&\Vert f-g\Vert_{A}\sum_{m\geq 0} \sum_{k=0}^{m+1} (2A^{-2})^{m+2}    \Vert f\Vert_{A}^{m-k+1}\Vert g\Vert_{A}^{k}\\
&\leq& \Vert f-g\Vert_{A}(2A^{-2}) %
\left(\sum_{l\geq 0} \sum_{k\geq 0} (2A^{-2})^{l} \Vert f\Vert_A^{l} (2A^{-2})^{k} \Vert g\Vert_{A}^k-1\right),
\end{eqnarray*}
where we have written $m=l+k-1$ which is non-negative precisely when $l$ and $k$ are not both zero. %
Thus, if $\|f\|_A\le A^2/2$ and $\|g\|_A\le A^2/2$, we deduce that
$$\Vert Q(\Sigma g)-Q(\Sigma f)\Vert_A \leq \Vert f-g \Vert_A \frac{2}{A^2}\left(\frac{1}{\left(1-\frac{2\Vert f\Vert_A}{A^2}\right)\left(1-\frac{2\Vert g\Vert_A}{A^2}\right)}-1\right).
$$
Setting $f=0$ gives us the estimate
$$\Vert Q(\Sigma g)\Vert_A \leq \frac{2\Vert g\Vert_A}{A^2}\left(\frac{1}{1-\frac{2\Vert g\Vert_A}{A^2}}-1\right),$$ as claimed.
 \end{proof}
\begin{cor} %
\label{cor:FisLipsch}Let $g\in \mathscr{A}^{(A)}_0$, %
assume that  $\Vert g\Vert_{A}<A^{2}/2$ and let $B\geq A+\Vert g\Vert_A$.  Let 
$W\in\mathscr{A}^{(B)}$. 
Assume that (\ref{eq:boundOnTau}) holds and moreover that $C_{0}/A<1/2$.  %
Let
\begin{eqnarray*}
F(g)&=&-W(X)+\Bigg((1\otimes\tau+\tau\otimes1)(Tr(\mathscr{J}\mathscr{D}\Sigma g))\\
&&-\left\{ W(X+\mathscr{D}\Sigma g)-W(X)+Q(\Sigma g)+\frac{1}{2}\mathscr{D}\Sigma g\#\mathscr{D}\Sigma g\right\} \Bigg)\\
&=&-W(X+\mathscr{D}\Sigma g)-\frac{1}{2}\mathscr{D}\Sigma g\#\mathscr{D}\Sigma g+(1\otimes\tau+\tau\otimes1)Tr\log(1+\mathscr{J}\mathscr{D}\Sigma g).
\end{eqnarray*}
(here we abbreviate $W(X_{1},\dots,X_{n})$ by $W(X)$, %
etc.) 

Then $F(g)$ is a well-defined function from $\{g \in \mathscr{A}_{0}^{(A)}: \Vert g\Vert_A < {A^2}/{2}\}$
 to $\mathscr{A}^{(A)}$. %
  Moreover $g\mapsto F(g)$
is locally Lipschitz on $\{ g : \Vert g\Vert_A < A^2 / 2\}$:
\begin{multline*}
\Vert F(g)-F(f)\Vert_{A}\leq
\Vert f-g \Vert_A \Bigg\{ \frac{2}{A^2}\left(\frac{1}{\left(1-\frac{2\Vert f\Vert_A}{A^2}\right)\left(1-\frac{2\Vert g\Vert_A}{A^2}\right)}+1\right)\\
+ \sum_j \Vert \partial_j W\Vert_{B\otimes_\pi B} + \frac{1}{2}\left(\Vert f\Vert_A + \Vert g\Vert_A\right) 
\Bigg\}
\end{multline*}
 and bounded:
\[
\Vert F(g)\Vert_{A}\leq
\Vert g \Vert_A \Bigg\{ \frac{2}{A^2}\left(\frac{1}{1-\frac{2\Vert g\Vert_A}{A^2}}+1\right)
+ \sum_j \Vert \partial_j W\Vert_{B\otimes_\pi B} + \frac{1}{2} \Vert g\Vert_A
\Bigg\} + \Vert W\Vert_A
\]

In particular, if $A$, $\rho$ and $W$ are such that %
\begin{equation}\label{eq:contractivityConditions}
\left\{ \begin{array}{l}A>4, \qquad 0<\rho \leq 1 \\
\Vert W\Vert_A<\frac{\rho}{12}
\\ \sum_j\Vert \partial_j W\Vert_{(A+\rho)\otimes_\pi (A+\rho)} < \frac{1}{8}
\end{array}\right. \end{equation}
 then $F$ takes the ball $\{ g: \Vert g\Vert_A <\frac{\rho}{4}\}$ to itself and is uniformly contractive with constant $\lambda \leq \frac{7}{8}$ on that ball.
\end{cor}
\begin{proof}
Note that $\Vert \mathscr{D}\Sigma g\Vert_A \leq \Vert g\Vert_A$ and so $\Vert X+\mathscr D\Sigma g\Vert \leq B$.  
Furthermore, for any $h$, $$\sum_i \Vert \mathscr{D}_i \Sigma h\Vert_A \leq \Vert h\Vert_A.$$ Indeed, if  $h=\sum_r\sum_{i_1,\dots,i_r} 
\alpha_{i_1,\dots,i_r} X_{i_1}\cdots X_{i_r}$, then since $A>1$
$$\sum_i \Vert \mathscr{D}_i \Sigma h\Vert_A = \sum_r \sum_{i_1,\dots,i_r} \sum_{q=1}^r |\alpha_{i_1,\dots,i_r} | r^{-1} \Vert 
X_{i_{q+1}}\cdots X_{i_r} X_{1}\cdots X_{i_{q-1}} \Vert_A \leq \frac{1}{A} \Vert h\Vert_A\le\Vert h\Vert_A .$$
The Lipschitz property then follows from Lemmas \ref{lemma:DefOfQ} and \ref{lipQ}, %
 as well as the estimates:
\begin{eqnarray*}
\Vert W(X+\mathscr{D}\Sigma g)-W(X+\mathscr{D}\Sigma f)\Vert_{A} & \leq & \sum_j \Vert\partial_j W\Vert_{B\otimes_{\pi}B}\Vert\mathscr{D}_j\Sigma g-\mathscr{D}_j\Sigma f\Vert_{A}\\
 & \leq & \sum_j \Vert\partial_j W\Vert_{B\otimes_{\pi}B}\Vert g-f\Vert_{A} 
\end{eqnarray*}
and
\begin{eqnarray*}
\left\Vert \frac{1}{2}\mathscr{D}\Sigma g\#\mathscr{D}\Sigma g-\frac{1}{2}\mathscr{D}\Sigma f\#\mathscr{D}\Sigma f\right\Vert _{A} & \le & \frac{1}{2}\Vert\mathscr{D}\Sigma g\#(\mathscr{D}\Sigma g-\mathscr{D}\Sigma f)\Vert_{A}\\
 &  & +\frac{1}{2}\Vert\mathscr{D}\Sigma f\#(\mathscr{D}\Sigma g-\mathscr{D}\Sigma f)\Vert_{A}\\
 &\leq & \frac{1}{2}(\sum_i \Vert \mathscr{D }_i\Sigma g\Vert_A) \max_{i} \Vert (\mathscr{D}_i \Sigma g - \mathscr{D}_i \Sigma f) \Vert_A 
 \\ & & +  \frac{1}{2} (\sum_i \Vert \mathscr{D }_i\Sigma f\Vert_A) \max_{i} \Vert (\mathscr{D}_i \Sigma g - \mathscr{D}_i \Sigma f) \Vert_A \\
 & \leq & \frac{1}{2}\Vert\mathscr{D}\Sigma g-\mathscr{D}\Sigma f\Vert_{A}\left(\Vert\mathscr{D}\Sigma g\Vert_{A}+\Vert\mathscr{D}\Sigma f\Vert_{A}\right) \\ 
 & \leq & \frac{1}{2}\Vert g-f\Vert_A \left( \Vert g\Vert_A + \Vert f \Vert_A \right).
\end{eqnarray*}
The estimate on $\Vert F(g)\Vert_A$ follows from the identity $F(0)=-W$, Lemmas \ref{lem:EstimateQm} and \ref{lemma:DefOfQ} and the Lipschitz property.

Assuming that  \eqref{eq:contractivityConditions} holds, so that $A>4$ and $\Vert f\|_A, \Vert g\Vert_A < \frac{\rho}{4} < \rho\leq 1$, we have that
$$\frac{2}{A^2}\left(\frac{1}{\left(1-\frac{2\Vert f\Vert_A}{A^2}\right)\left(1-\frac{2\Vert g\Vert_A}{A^2}\right)}+1\right)
< \frac{1}{8}\left(\frac{64}{49} +1 \right) < \frac{1}{2}$$ and so the Lipschitz constant of $F$ is bounded by 
$\frac{1}{2} + \frac{1}{4} +\sum_j \Vert \partial_j W\Vert_{(A+\rho)\otimes_\pi (A+\rho)}< \frac {1}{2}+\frac{1}{4} + \frac{1}{8}=\frac{7}{8}$.  Also,
$$\Vert F(g)\Vert < \frac{\rho}{4} \left\{ \frac{1}{8}\left(\frac{8}{7}+1\right)+\frac{1}{8} + \frac{1}{8}\right\} + \Vert W\Vert_A <
\left(\frac{2}{3}  +\frac{1}{3}\right)\frac{\rho}{4}=\frac{\rho}{4}  $$
so that the image of the ball $\{g:\Vert g\Vert_A<\frac{\rho}{4}\}$ is contained in that ball.
\end{proof}

\subsection{Existence of $g$.}

We will now consider the case that $\tau$ is the semicircle law.
Then (\ref{eq:boundOnTau}) holds with $C_{0}=2$. We remind the reader that $\Pi, \Sigma,$ are defined in section \ref{secdef} whereas $\mathscr S$ is defined in Definition \ref{defS}. %
\begin{prop}
\label{prop:gExists}Let $\tau$ be the semicircle law.  Assume that for some $A$ and $\rho$, $W\in\mathscr{A}_0^{(A)}$ is cyclically symmetric and that 
 conditions \eqref{eq:contractivityConditions} are satisfied, i.e., 
\begin{equation*}
\left\{ \begin{array}{l}A>4,\qquad 0<\rho\leq 1 \\
\Vert W\Vert_A<\frac{\rho}{12}
\\ \sum_j \Vert \partial_j W\Vert_{(A+\rho)\otimes_\pi (A+\rho)} < \frac{1}{8}
\end{array}\right. \end{equation*}
Then there exists $\hat{g}$ and $g=\Sigma\hat{g}$ (so that $\hat{g}=\mathscr{N}g$)  %
with the following properties:\\
(i) Both $\hat{g}$ and $g$ belong to the completion of $\mathscr{A}_{0}$
with respect to the norm $\Vert\cdot\Vert_{A}$ and $\Vert g\Vert_A \leq \rho/4$, $\Vert g\Vert_A \leq 3 \Vert W\Vert_A $;\\
(ii) $\hat{g}$ satisfies the equation $\hat{g}=\mathscr{S}\Pi F(\hat{g})$\\
(iii) $\hat{g}$ and $g$ depend analytically on $W$, in the following
sense: if the maps $\beta\mapsto W_{\beta}$ are analytic, then also
the maps $\beta\mapsto\hat{g}(\beta)$ and $\beta\mapsto g(\beta)$ are analytic,
and $g\to0$ if $\Vert W\Vert_{A}\to0$\\
(iv) $g$ satisfies the equation
\begin{multline}
\mathscr{S}\Pi\mathscr{N}g=- W(X)+\mathscr{S}\Pi\Bigg((1\otimes\tau+\tau\otimes1)(Tr(\mathscr{J}\mathscr{D}g))-\\
\left\{ W(X+\mathscr{D}g)-W(X)+Q(g)+\frac{1}{2}\mathscr{D}g\#\mathscr{D}g\right\} \Bigg)\\
=\mathscr{S}\Pi\left[-W(X+\mathscr{D} g)-\frac{1}{2}\mathscr{D} g\#\mathscr{D} g+(1\otimes\tau+\tau\otimes1)Tr\log(1+\mathscr{J}\mathscr{D}g)\right].\label{eq:freeMongeAmpereEquationNdim}
\end{multline}
or, equivalently,
\begin{multline}
-\mathscr{S}\Pi\mathscr{N}g+\mathscr{S}\Pi(1\otimes\tau+\tau\otimes1)(Tr(\mathscr{J}\mathscr{D}g))=W(X)\\
+\mathscr{S}\Pi\left\{ W(X+\mathscr{D}g)-W(X)+Q(g)+\frac{1}{2}\mathscr{D}g\#\mathscr{D}g\right\} .\label{eq:Forg}
\end{multline}
\end{prop}
\begin{proof} Observe that Equation \eqref{eq:freeMongeAmpereEquationNdim} is equivalent to
$$\mathscr{S}\Pi\mathscr{N}g=\mathscr{S}\Pi F(\mathscr{N}g)$$
with $F$ the function defined in Corollary \ref{cor:FisLipsch}, where we used that $W$ is cyclically symmetric and in $\mathcal A_0^{(A)}$. %
Existence of $g$ follows, essentially, from the Implicit function
Theorem \ref{theoimplicit}. We prefer to repeat its proof here for completeness. We set
$\hat{g}_{0}=W(X_{1},\dots,X_{n})$, %
and
for each $k>0$,
\[
\hat{g}_{k}=\mathscr{S}\Pi F(\hat{g}_{k-1}),%
\]%

Since $\mathscr{S}\Pi$ is a linear contraction, the last part of Corollary~\ref{cor:FisLipsch} implies that under our hypothesis, $\mathscr{S}\Pi\circ F$ is uniformly contractive with constant $7/8$ on the ball $B=\{\hat g: \Vert \hat{g}\Vert_A < \frac{\rho}{4}\}$, and takes this ball to itself.  It follows that $\hat{g}_k\in B$ for all $k$, and moreover 
\[
\Vert\hat{g}_{k}-\hat{g}_{k-1}\Vert_{A}=\Vert \mathscr{S}\Pi F(\hat{g}_{k-1})-\mathscr{S}\Pi F(\hat{g}_{k-2})\Vert<\frac{7}{8}\Vert\hat{g}_{k-1}-\hat{g}_{k-2}\Vert_{A},
\]
so that $\hat{g}_{k}$ converges in $\Vert\cdot\Vert_{A}$ to a fixed
point $\hat{g}$. Since $\hat{g}\in \mathscr{A}_{0}^{(A)}$, so does $g=\Sigma\hat{g}$. 

Note that if conditions \eqref{eq:contractivityConditions} are satisfied for some $\rho$, and $\Vert W\Vert_A < \rho'/12$ for some $\rho'< \rho$, then 
 conditions \eqref{eq:contractivityConditions} are again satisfied with $\rho'$ in place of $\rho$ (indeed, $\Vert \partial_j W\Vert_{(A+\rho')\otimes_\pi (A+\rho')}
 \leq \Vert \partial_j W\Vert_{(A+\rho)\otimes_\pi (A+\rho)}$).  Since by construction $\Vert g\Vert_A \leq \Vert \hat{g}\Vert_A \leq \frac{\rho}{4}$, 
 we conclude that in fact $\Vert g\Vert_A \leq \frac{\rho'}{4}$, for any $\rho'> 12 \Vert W\Vert_A$.  It follows that $\Vert g \Vert_A \leq 3 \Vert W\Vert_A$.
 
 This proves (i) and (ii). 

Since (assuming that $\beta\mapsto W_{\beta}$ is analytic) each iterate
$\hat{g}_{k}$ is clearly analytic in $\beta$ and the convergence
of the Banach-space valued function $\hat{g}_{k}(\beta)\in{\mathscr{A}}_{0}^{({A})}$ %
is uniform on any compact disk inside $|\beta|<\beta_{0}$, it follows
from the Cauchy integral formula that the limit is also an analytic
function. 

Part (iv) follows from the definition of $F$ and $\Sigma$.
\end{proof}
We leave the following to the reader (note that $\hat{g}$ is clearly
cyclically symmetric, and it is not hard to see that this implies
that also $g$ is cyclically symmetric):
\begin{prop}
Assume that $W=W^{*}$ and that $W$ satisfies the hypothesis of Proposition~\ref{prop:gExists}.
Let $g$ be the solution to (\ref{eq:Forg}) constructed in %
 Proposition \ref{prop:gExists}. Then $g$ belongs to the closure
of cyclically symmetric polynomials and also satisfies $g=g^{*}$.
\end{prop}

\subsection{The map $f=\mathscr{D}g$ satisfies Equation (\ref{eq:Main-1}).}

From now on, we will assume that $\tau$ is the semicircle law.
\begin{thm}
\label{thm:fExists}Let  $A>A'>4$. Then there exists %
a constant $C(A,A')>0$ depending only on $A$, $A'$ so
that whenever $W=W^{*}\in\mathscr{A}^{(A)}$ satisfies $\Vert W\Vert_{A+1}<C(A,A')$, %
there exists an $f\in(\mathscr{A}^{(A')})^{n}$ which satisfies (\ref{eq:Main-1}).
In addition, $f=\mathscr{D}g$ for some $g\in\mathscr{A}^{(A')}$.
The solution $f=f_{W}$ satisfies $\Vert f_{W}\Vert_{A'}\to0$ as
$\Vert W\Vert_{A}\to0$. Moreover, if $W_{\beta}$ is a family which 
is analytic in $\beta$ then also the solutions $f_{W_{\beta}}$ are
analytic in $\beta$. \end{thm}
\begin{proof}
Let $A_1$ be so that $4<A_1 < A$, and let  $$c(A_1,A) = \frac{1} {  e A_1 \log(A/A_1)}  = \sup_{\alpha\geq 1} \alpha {A_1^{-1}}  (A/A_1)^{-\alpha}.$$
Then if $g=\sum \lambda(q) q$, we obtain the estimate
\begin{eqnarray*}
\Vert\mathscr{D}g\Vert_{A_{1}} &\leq& \sum |\lambda(q)| \mbox{deg}(q) A_1^{\textrm{deg}(q)-1}\\
 &\leq& \sum |\lambda(q)| \mbox{deg}(q)A^{\textrm{deg}(q)}  {A_1^{-1}} (A/A_1)^{-\textrm{deg}(q)}  \\
 &\leq& \sum |\lambda(q)|A^{\textrm{deg}(q)} c(A,A_1) \leq  c(A_1,A) \|g\|_{A}
\end{eqnarray*}
Similarly,  if we set $c'(A_1,A) = \sup_{\alpha\geq 1}  \alpha^2 {A_1^{-2}}  (A/A_1)^{-\alpha} < \infty$, then 
\begin{eqnarray} \label{bbJ}
\Vert \mathscr{J} f\Vert_{A_{1}\otimes_{\pi}A_{1}} &=& \max_i \sum_j \Vert \partial_j \mathscr{D}_i g\Vert_{A_{1}\otimes_{\pi}A_{1}} \\ \nonumber
&\leq& \sum |\lambda(q)| (\textrm{deg}(q))^{2} A_1^{\textrm{deg}(q)-2}\nonumber\\ \nonumber
&\leq&  c'( A_1,A)  \Vert g\Vert_{A}.
\end{eqnarray}
Finally, by the same computation as in the proof of Lemma~\ref{eq:contractivityConditions},
\[
\sum_i \Vert \partial_i W\Vert_{ (A_{1}+1) \otimes_{\pi} (A_{1}+1)}\leq c(A_1+1,A+1) \Vert W\Vert_{A+1}.
\]

Thus we can choose $C(A,A')>0$ so that if $\Vert W\Vert_{A+1} < C(A,A')$, then the hypothesis of Proposition~\ref{prop:gExists} is satisfied (with $\rho=1$), and there exists some $g$ which satisfies
\begin{multline}
-\mathscr{S}\Pi\mathscr{N}g+\mathscr{S}\Pi(1\otimes\tau+\tau\otimes1)(Tr(\mathscr{J}\mathscr{D}g))  \\
=\mathscr{S}\Pi W(X)+\mathscr{S}\Pi\left\{ W(X+\mathscr{D}g)-W(X)+Q(g) \label{eqn:Main-3}%
+\frac{1}{2}\mathscr{D}g\#\mathscr{D}g\right\} .
\end{multline}
Furthermore, we may assume (by choosing a perhaps smaller $C(A,A')$) that $f=\mathscr{D}g$ satisfies $\Vert \mathscr{J}f\Vert_{A'\otimes_\pi A'} < 1$.  
To deduce that $g$ satisfies (\ref{eq:Main-1}), we shall apply $\mathscr{D}$ on both sides of the above equality.
To this end, 
note that if $p = (p_1,\dots,p_n)$ with $p_i$ monomials, then 
\begin{eqnarray*}
(\mathscr{D} (p \# p))_i &=& \mathscr{D}_i (\sum_j p_j  p_j)  \\
&=& \sum_j \left[ \sum_{p_j p_j 
= p_j AX_i B} B p_j  A + \sum_{p_j p_j =AX_i B p_j } B p_j A \right]
\\ &=& 2 \sum_j \sum_{p_j = A X_i B} B p_j A
\\ &=& 2 \sum_j \sigma (\mathscr{J}p)_{ji} \# p_j  \qquad \textrm{(here $\sigma(a\otimes b)=b\otimes a$)}
\end{eqnarray*}
Now if we assume that $p=\mathscr{D}w$ for some $w$, then by \eqref{eq:signaVSadj} in the proof of (iii) in Lemma~\ref{change}, 
we find that
$$(\mathscr{D} (p \# p))_i = 2 \sum_j  \sigma (\mathscr{J}p)_{ji} \# p_j  =2 \sum_j  (\mathscr{J}p)_{ij} \# p_j   =  2(\mathscr{J}p) \# p.$$

Using this identity, applying $\mathscr{D}$ to both sides of \eqref{eqn:Main-3} and noting that for any $h$,
$\mathscr{D}\mathscr{S}\Pi h=\mathscr{D}h$,   gives us
\begin{eqnarray*}
\mathscr{D}\left\{ (1\otimes\tau+\tau\otimes1)\left(Tr\left[\mathscr{J}\mathscr{D}g\right]\right)-\mathscr{N}g\right\} %
 & = & \mathscr{D}(W(X+\mathscr{D}g))+\mathscr{D}Q(g)+\mathscr{J}\mathscr{D}g\# \mathscr{D}g,
\end{eqnarray*}
which, according to Lemma \ref{lem:IfWeCanSolveFreeMongeAmpere} is
equivalent to (\ref{eq:Main-1}). 
\end{proof}
Since $Y_{W}=X+f_{W}\rightarrow X$  as $\Vert W\Vert_A \to0$, it
follows that for sufficiently small $\Vert W\Vert_{A}$, $\Vert Y\Vert$ is bounded by $A$ and so the law of
$Y$ is the unique solution to the Schwinger-Dyson equation with self-adjoint potential
$V=\frac{1}{2}\sum X_{j}^{2}+W$, and thus the law of $Y$ is exactly
$\tau_{V}$ (by \cite[Theorem 2.1]{guionnet-edouard:combRM}  or alternatively %
since, for $\Vert W\Vert_{A}$ sufficiently small, $V$
is then $(c,A)$ convex and the Schwinger-Dyson equation has a unique
solution \cite{alice-shlyakhtenko-freeDiffusions}). In particular,
$W^{*}(Y_{1},\dots,Y_{n})\cong W^{*}(\tau_{V})$. We thus obtain a
free analog of Brenier's theorem:
\begin{thm}  %
\label{thm:BrennierType}Let $A>A'>4$, and let $X_{1},\dots,X_{n}\in(M,\tau)$
be semicircular variables. Then there exists a universal constant
$C=C(A,A')>0$ so that whenever $W\in\mathscr{A}^{(A+1)}$ satisfies
$\Vert W\Vert_{A+1}<C$, there is an element $G\in\mathscr{A}^{(A')}$
so that
\[
(Y_{1},\dots,Y_{n})=(\mathscr{D}_{1}G,\dots,\mathscr{D}_{n}G)\in\mathscr{A}^{(A')}
\]
has law $\tau_{V}$, $V=\frac{1}{2}\sum X_{j}^{2}+W$. 

Moreover, the Hessian $\mathscr{J}\mathscr{D}G$ is a strictly positive
element of $M_{n\times n}(M\bar{\otimes}M^{op})$.

In particular, there are trace-preserving injections $C^{*}(\tau_{V})\subset C_{\operatorname{red}}^{*}(\mathbb{F}_{n})$
and $W^{*}(\tau_{V})\subset L(\mathbb{F}_{n})$.

If the map $\beta\mapsto W_{\beta}$ is analytic, then $Y_{1},\dots,Y_{n}$
are also analytic in $\beta$. Furthermore, $\Vert Y_{j}-X_{j}\Vert_{A'}$ vanishes as  $\Vert W\Vert_{A+1}$ goes to zero.\end{thm}
\begin{proof}
We use Theorem \ref{thm:fExists} and set $G=\frac{1}{2}\sum X_{j}^{2}+g$
and the discussion before the statement of the present theorem. The
only thing left to prove is that the Hessian of $G$ is strictly positive.
But the Hessian of $G$ is given by
\[
\mathscr{J}\mathscr{D}G=1+\mathscr{J}\mathscr{D}g.
\]
In the proof of Theorem \ref{thm:fExists}, we have chosen $C$ in such a way that $\Vert \mathscr{J}\mathscr{D}g\Vert_{A'\otimes_\pi A'}<1$, which means that $\mathscr{J}\mathscr{D}G$ is strictly positive.
\end{proof}

\section{Some applications.\label{sec:Applications.}}

\subsection{Analyticity of the solution to the Schwinger-Dyson equation.}

Under the hypothesis of Theorem \ref{thm:BrennierType}, we deduce
that if $\beta\to W_{\beta}$ is an analytic family of potentials
and $V_{\beta}=\frac{1}{2}\sum X_{j}^{2}+W_{\beta}$, then there exists
an analytic family of elements $Y_{\beta}$ whose law $\tau_{\beta}$
satisfies the Schwinger-Dyson equation for $V_{\beta}$ . We deduce
the following corollary, which was already proved in \cite{alice-shlyakhtenko-freeDiffusions},
by a different method:
\begin{cor}
Let $P\in\mathscr{A}$ be a fixed polynomial. Then $\tau_{\beta}(P)$
is analytic in $\beta$ in a neighborhood of the origin. %
\end{cor}

\subsection{Isomorphism results.}

Applying the implicit function theorem, we can improve Theorem \ref{thm:BrennierType}
at the expense of possibly choosing a smaller bound on $W$. Indeed,
we let $(Y_{1},\dots,Y_{n})=f=\mathscr{D}g$ be as in Theorem \ref{thm:BrennierType},
we can always assume that $Y_{j}(0,\dots,0)=0$ by replacing each
$Y_{j}$ with $Y_{j}-Y_{j}(0,\dots,0)$. Thus we may assume that $f(0)=0$. 
\begin{thm}  %
\label{thm:isomorphism}Let $A>A'>4$, and let $X_{1},\dots,X_{n}\in(M,\tau)$
be semicircular variables. Then there exists a universal constant
$C=C(A,A')>0$ so that whenever $W\in\mathscr{A}^{(A+1)}$ satisfies
$\Vert W\Vert_{A+1}<C$, there is an element $G\in\mathscr{A}^{(A')}$
so that:
\begin{enumerate}
\item If we set $Y_{j}=\mathscr{D}_{j}G$, then $Y_{1},\dots,Y_{n}\in\mathscr{A}^{(A')}$
has law $\tau_{V}$, with $V=\frac{1}{2}\sum X_{j}^{2}+W$;
\item $X_{j}=H_{j}(Y_{1},\dots,Y_{n})$ for some $H\in\mathscr{A}^{(A')}$;
\item the Hessian $\mathscr{J}\mathscr{D}G$ %
 is a strictly positive element
of $M_{n\times n}(M\bar{\otimes}M^{op})$. 
\end{enumerate}
In particular, there are trace-preserving isomorphisms
\[
C^{*}(\tau_{V})\cong C^{*}(X_{1},\dots,X_{n}),\qquad W^{*}(\tau_{V})\cong L(\mathbb{F}_{n}).
\]

\end{thm}
\begin{proof} Fix $B\in (A',A)$.  By Theorem
\ref{thm:BrennierType}, we can write 
$$Y= X+  f(X)$$
with $\beta=\|f\|_{B}$  going to zero  as $C=\Vert W\Vert_{A+1}$ goes to zero. 
Moreover, as $X$ are semicircular variables, their norms are bounded 
by $2$ and hence for $C$ (and thus $\beta$) small enough, $Y$ is bounded by $A'$ so that
 we can apply Corollary \ref{corimplicit} with $A'<B$ to find a function $H\in (\mathcal A^{(A')})^n$ so that
$X=H(Y)$. 

\end{proof}

\subsubsection{Isomorphism class of algebras with analytic conjugate variables that
are close to the generators.}
\begin{cor}
Let $X_{1},\dots,X_{n}\in(M,\tau)$ be generators of a von Neumann
algebra $M$ and assume that the Fisher information $\Phi^{*}(X_{1},\dots,X_{n})$
is finite. Assume moreover that $\xi_j^* = \xi_{j}=\partial_{j}^{*}(X_{j})$ %
belongs to $\mathscr{A}^{(A+1)}$ for some $A>4$. 

Then there exists a universal constant $C=C(A)$ so that if $\Vert\xi_{j}-X_{j}\Vert_{A+1}<C$,
then $M\cong L(\mathbb{F}_{n})$ and also $C^{*}(X_{1},\dots,X_{n})\cong C^{*}(S_{1},\dots,S_{n})$
where $S_{1},\dots,S_{n}$ are free semicircular variables.\end{cor}
\begin{proof}
Letting $V=\frac{1}{2}\Sigma\left(\sum_{j=1}^{n}X_{j}\xi_{j}+\xi_j X_j \right)$, one gets
that $\xi_{j}=\mathscr{D}_{j}V$ %
(cf. \cite{voiculescu:cyclicGradients}).
Thus we can write $V=\frac{1}{2}\sum X_{j}^{2}+W$; then for small
enough $\Vert\xi_{j}-X_{j}\Vert_{A+1}$, $W=W^*$ satisfies the hypothesis %
of Theorem \ref{thm:isomorphism}.
\end{proof}

\subsubsection{Voiculescu's conjecture with polynomial potentials.}

The following corollary is a partial answer to a conjecture of Voiculescu
\cite[p. 240]{voiculescu:conjectureAboutPotentials} (the full statement
of Voiculescu's conjecture is that the isomorphism should hold for
arbitrary polynomials $W$ and arbitrary $\beta$;
 however, even uniqueness
of a free Gibbs state is unknown in that generality!)
\begin{cor}
Let $A>4$ and let $W\in\mathscr{A}^{(A)}$ be a  self-adjoint %
power series. Let
$V_{\beta}=\sum X_{j}^{2}+\beta W$, and let $\tau_{\beta}$ be a
trace satisfying the Schwinger-Dyson equation with potential $V_{\beta}$.
Then there exists a $\beta_{0}>0$ so that $W^{*}(\tau_{\beta})\cong W^{*}(\tau_{0})\cong L(\mathbb{F}_{n})$
and $C^{*}(\tau_{\beta})\cong C^{*}(\tau_{0})=C^{*}(S_{1},\dots,S_{n})$
for all $-\beta_{0}<\beta<\beta_{0}$ (here $S_{1},\dots,S_{n}$ is
a free semicircular system).
\end{cor}

\subsubsection{$q$-deformed free group factors.}

In \cite{Speicher:q-comm}, Bozejko and Speicher have introduced a
family of von Neumann algebras $\Gamma_{q}(\mathbb{R}^{n})$, which
are ``$q$-deformations'' of free group factors. For $q=0$, $\Gamma_{q}(\mathbb{R}^{n})\cong L(\mathbb{F}_{n})$,
but the question of whether an isomorphism like this holds for $q\neq0$
remained open. Despite being much-studied (\cite{Speicher:q-comm,speicher:q-GaussianProcesses,nou:qNonInjective,sniady:q-RM-model,sniady:q-nonGamma,shlyakht:someEstimates,ricard:QFactoriality,shlyakht:lowerEstimates,kennedy-nica:exactQDeformed,dabrowski:SPDE,avsec:Q}
is an incomplete list of results about these factors), the question
of the isomorphism class of these algebras for all values of $q$
remains elusive. Nonetheless, we are able to settle it for small $q$
(depending on $n<\infty$):
\begin{cor}
Let $\Gamma_{q}(\mathbb{R}^{n})$ be the von Neumann algebra generated
by a $q$-semicircular system $S_{1}^{(q)},\dots,S_{n}^{(q)}$, $n\in\{2,3,\dots\}$.
Then there exists a $0<q_{0}=q_{0}(n)$ depending on $n$, so that
$\Gamma_{q}(\mathbb{R}^{n})\cong\Gamma_{0}(\mathbb{R}^{n})\cong L(\mathbb{F}_{n})$
and $C^{*}(S_{1}^{(q)},\dots,S_{n}^{(q)})=C^{*}(S_{1}^{(0)},\dots,S_{n}^{(0)})$
for all $|q|<q_{0}$.\end{cor}
\begin{proof}
This follows from the fact \cite[Theorem 34]{dabrowski:SPDE} that the conjugate  %
variables $\xi_{j}=\xi_j^*=\mathscr{D}V$ to $(S_{1}^{(q)},\dots,S_{n}^{(q)})$ exist and %
that for some $A>5$, $\Vert\xi_{j}-X_{j}\Vert_{A}\to0$ as $q\to0$. %
\end{proof}

\subsubsection{Free entropy.}

Note that under the assumptions of Theorem \ref{thm:isomorphism},
we end up expressing $X_{j}$'s as a convergent power series in $n$
semicircular variables $S_{1},\dots,S_{n}$. By the change of variables
formula of Voiculescu \cite{dvv:entropy2}, we get:
\begin{cor}
\label{cor:ChangeOfEntropyFormula}Let $V=\frac{1}{2}\sum X_{j}^{2}+W$,
and assume the hypothesis of Theorem \ref{thm:isomorphism}. Let $\tau_{V}$
be the free Gibbs state with potential $V$. Then
\[
\chi(\tau_{V})=\chi(\tau)+\tau\otimes\tau [\operatorname{Tr}(\log(\mathscr{J}\mathscr{D}G))] %
\]
where $\tau$ is the semicircle law and $G$ is as in Theorem \ref{thm:isomorphism}.
\end{cor}

\subsection{Monotone transport for random matrix models.}

Let $V=\frac{1}{2}\sum X_{j}^{2}+W$ and $\tau_{V}$ be the free Gibbs
state with potential $V$.  Assume that $W$ satisfies the
hypothesis of Theorem \ref{thm:isomorphism} for some $A>A'>4$ and let $F=(F_1,\ldots, F_n)={\mathscr D} G$ be the map %
in $(\mathcal A^{(A')})^n$
constructed in that theorem.  Let us assume that $W$ is small enough so that $\Vert F\Vert_{A'} \leq A$.  
We let
\[
V(X_{1},\ldots,X_{n})=\frac{1}{2}\sum X_{j}^{2}+W(X_{1},\cdots,X_{n}).
\]
on $\{\max_i \|X_i\|\le A'\}$ and $ V=+\infty$ otherwise.  Note that $V$ is strictly convex if $\|W\|_{A+1}$ is small enough.

Let $\mu_{V}^{(N)}$ be the measure on $(M_{N\times N}^{sa})^{n}$
given by
\[
d\mu_{V}^{(N)}=\frac{1}{Z_{N}^{V}}\exp(-N{\rm Tr}(V(A_{1},\dots,A_{n})))dA_{1}\dots dA_{n}
\]
where we set $dA_{1}\dots dA_{n}=\textrm{Lebesgue measure}$ and $M_{N\times N}^{sa}$ the set of $N\times N$ self-adjoint matrices. Consider
also the Gaussian measure
\[
d\mu^{(N)}=\frac{1}{Z_{N}}\exp\left[-N{\rm Tr}\left(\frac{1}{2}\sum A_{j}^{2}\right)\right]dA_{1}\dots dA_{n}.
\]

We denote by $\chi:\mathbb{R}\to[-A',A']$
a smooth cutoff function so that $\chi(x)=x$ for $x\in[-4,4]$, $|\chi(x)|\le A'$
and $\chi$ has a uniformly bounded derivative. 
On the space of $n$-tuples of $N\times N$ self-adjoint matrices
$(M_{N\times N}^{sa})^{n}$ with norm bounded by $4$, consider
the functional calculus map
\[
\tilde{F}^{(N)}:(A_{1},\dots,A_{n})\mapsto(F_{j}(\chi(A_{1}),\dots,\chi(A_{n})))_{j=1}^{n}.
\]
Finally, let $F^{(N)}=\nabla\Psi^{(N)}$ be the unique monotone transport
map on $(M_{N\times N}^{sa})^{n}$  so that
\[
F^{(N)}_{*}\mu^{(N)}=\mu_{V}^{(N)}.
\]
 
\begin{thm}
\label{thrm:MatricialTransport}Let $A>A'>4$ as above.  Then there exists a constant $c$ (smaller than that of 
Theorem~\ref{thm:isomorphism}) so that for any $\|W\|_{A+1}<c$, the following statement hold:\\
(i) With the above notation, 
\[
\lim_{N\to\infty}\int\frac{1}{N}Tr\left(\left[F^{(N)}-\tilde{F}^{(N)}\right]^{2}\right)d\mu^{(N)}=0.
\]
In other words, as $N\to\infty$, the ``entrywise'' monotone transport
$F^{(N)}$ is well-approximated as $N\to\infty$ by the ``matricial''
functional calculus map $\tilde{F}^{(N)}$.\\
(ii) Assume in addition that $W$ is $(c,A)$-convex for some $c>0$.
Then $\mathscr{J}F\leq1$. Furthermore, $\frac{1}{N}Tr\left(\left[F^{(N)}-\tilde{F}^{(N)}\right]^{2}\right)$ vanishes 
almost surely as $N$ goes to infinity. %
\end{thm}
\begin{proof}
Let us put, for a probability measure $\nu$ on $\mathbb{R}^{M}$
with density $p(x_{1},\dots,x_{M})dx_{1}\dots dx_{M}$, and $U:\mathbb{R}^{M}\to\mathbb{R}\cup \{+\infty\}$ %
a convex function, 
\begin{eqnarray*}
H(\nu) & = & \int p(x_{1},\dots,x_{M})\log p(x_{1},\dots,x_{M})dx_{1}\cdots dx_{M}\\
H_{U}(\nu) & = & \int p(x_{1},\dots,x_{M})\log p(x_{1},\dots,x_{M})dx_{1}\dots dx_{M}+\int Ud\nu.
\end{eqnarray*}

If $h$ is an invertible transformation with positive-definite Jacobian
$\operatorname{Jac}h$, then
\begin{eqnarray*}
H(h_{{*}}\nu) & = & H(\nu)-\int\log\det\operatorname{Jac}h\ d\nu\\ %
 & = & H(\nu)-\int Tr\log\operatorname{Jac}h\ d\nu.
\end{eqnarray*}
Both entropies $H_{U}(\nu)$ and $H(\nu)$ are convex functions of
the density of $\nu$. Moreover, $H_U(h_{*}\nu)$ and $H(h_{*}\nu)$ are convex in $h$ if $U$ is convex. %

It is well-known that $H_{U}(\nu)$ %
is minimized precisely by the
Gibbs measure with potential $U$: %
\[
H_{U}(\nu)=\inf_{\nu'}H_{U}(\nu')\iff d\nu(x_{1},\dots,x_{M})=\frac{1}{Z}\exp(-U(x_{1},\dots,x_{M}))dx_{1}\dots dx_{M}.
\]

Similarly, for $V\in\mathscr{A}$, let $\chi(\tau)$ be Voiculescu's
microstates free entropy, and put $\chi_{V}(\tau)=\chi(\tau)-\tau(V).$ 

Then using \cite{dvv:entropysurvey,alice-shlyakhtenko-freeDiffusions,alice:StFlour}
one has
\[
\chi_{V}(\tau)=\sup_{\tau'}\chi_{V}(\tau')\iff\tau\textrm{ is the free Gibbs law with potential }V.
\]

Moreover, if we set $V^{(N)}(A_{1},\dots,A_{n})=NTr(V(A_{1},\dots,A_{N}))$, %
then following \cite[proof of Theorem 5.1]{alice-shlyakhtenko-freeDiffusions}
or \cite[Section 3.3]{guionnet-edouard:combRM} we see that $\mu_{V}^{(N)}(\frac{1}{N}{\rm Tr}(P))$
converges towards $\tau_{V}(P)$ for all polynomials $P$, in particular
the limit does not depend on the cutoff   provided $\|W\|_{A+1}$  is  %
small enough. In fact,
as $\mu_{V}^{(N)}$ has a strictly log-concave density,  Brascamp-Lieb
inequalities allow one to show that the matrices under $\mu_{V}^{(N)}$
are bounded by $4$ with overwhelming probability. Moreover, by definition of the entropy, we find that
\begin{eqnarray*}
\chi_{V}(\tau_{V}) & = & \lim_{N\to\infty}(
\frac{1}{N^{2}}\log Z_{N}^{V}+\frac{n}{2}\log N )
\\
 & = & \lim_{N\to\infty}(\frac{n}{2}\log N-\frac{1}{N^{2}}H_{V^{(N)}}(\mu_{V}^{(N)}))
\end{eqnarray*}
Furthermore, we claim that, if $\tilde\mu_V^N=\tilde F^{(N)} _{*}\mu^{(N)}$,we have
\begin{equation}\label{thenet}
\lim_{N\to\infty}\frac{1}{N^{2}}H_{V^{(N)}}(\tilde{\mu}_{V}^{(N)})-\frac{n}{2}\log N=-\chi_{V}(\tau_{V}).
\end{equation}
To see this, we first note that
\[
\lim_{N\to\infty}\frac{1}{N^{2}}\int V^{(N)}d\tilde{\mu}_{V}^{(N)}=\tau_{V}(V).
\]
Moreover,
\[
\frac{1}{N^{2}}H(\tilde{\mu}_{V}^{(N)})=\frac{1}{N^{2}}H(\mu^{(N)})-\mathbb{E}_{\mu^{(N)}}\left[\frac{1}{N^{2}}Tr\log\operatorname{Jac}\tilde{F}^{(N)}\right].
\]
On the other hand (see e.g. \cite{alice:StFlour}), because of concentration
phenomena, as $\mu^{(N)}$ has a strictly log-concave density, if
we take $\tau$ to be the semicircle law, then
\[
\mathbb{E}_{\mu^{(N)}}\left[\frac{1}{N^{2}}Tr\log\operatorname{Jac}\tilde{F}^{(N)}\right]\to\tau\otimes\tau(Tr\log\mathscr{J}F).
\]
Thus
\[
\frac{1}{N^{2}}H(\tilde{\mu}_{V}^{(N)})-\frac{n}{2}\log N\to-\chi(\tau)-\tau\otimes\tau(Tr\log\mathscr{J}F)=-\chi(\tau_{V})
\]
so that using Corollary \ref{cor:ChangeOfEntropyFormula} gives \eqref{thenet}. As a consequence, we have proved that
$$\lim_{N\rightarrow\infty} \frac{1}{N^{2}}(H_{V^{(N)}}(\tilde{\mu}_{V}^{(N)})-H_{V^{(N)}}(\mu_{V}^{(N)}))=0\,.$$
By convexity of entropy as a function of the transport maps, for $\varepsilon\in[0,1]$, %
\begin{eqnarray*}
&&H_{V^{(N)}}\left(\left[(1-\varepsilon)F^{(N)}+\varepsilon\tilde{F}^{(N)}\right]_{{\#}}\mu^{(N)}\right)-H_{V^{(N)}}(\mu_{V}^{(N)})\\
&\leq& (1-\varepsilon)H_{V^{(N)}}(F^{(N)}_{*}\mu^{(N)})+\varepsilon H_{V^{(N)}}(\tilde{F}^{(N)}_{*}\mu^{(N)})-H_{V^{(N)}}(\mu_{V}^{(N)})\\
&=& \varepsilon(H_{V^{(N)}}(\tilde{\mu}_{V}^{(N)})-H_{V^{(N)}}(\mu_{V}^{(N)})).%
\end{eqnarray*}
On the other hand, let $\Delta^{(N)}=F^{(N)}-\tilde{F}^{(N)}$ and
set
\begin{eqnarray*}
D_{N} & = & \partial_{\varepsilon}\Big|_{\varepsilon=0}H_{V^{(N)}}\left(\left[F^{(N)}+\varepsilon\Delta^{(N)}\right]_{{{*}}}\mu^{(N)}\right).
\end{eqnarray*}
Since $H_{V^{(N)}}\left(\left[(1-\varepsilon)F^{(N)}+\varepsilon\tilde{F}^{(N)}\right]_{{{*}}}\mu^{(N)}\right)$
has an absolute minimum at $\varepsilon=0$, we deduce that $D_{N}\geq0$. 

Furthermore, assuming that $\|W\|_{A+1}$ is small enough so that on the set  $\{\max_{1\le i\le n}\|A_{i}\|_{\infty}\le A\}$
the Hessian $\operatorname{Hess}(V^{(N)})$ is bounded from below
by $cN$  %
times the identity operator for some $c>0$, and noting that by definition the image of $\tilde F^{(N)}$
is bounded by $A$ whereas the image of $F^{(N)}$ composed with $\chi$ is also bounded by $A$, %
\begin{eqnarray}
 &  & H_{V^{(N)}}\left(\left[(1-\varepsilon)F^{(N)}+\varepsilon\tilde{F}^{(N)}\right]_{{{*}}} %
 \mu^{(N)}\right)-\varepsilon D_{N}-H_{V^{(N)}}(\mu_{V}^{(N)})\nonumber \\
 & = & \int_{0}^{\varepsilon}(\varepsilon -t)\partial_{t}^{2}H_{V^{(N)}}\left(\left[F^{(N)}+t\Delta^{(N)}\right]_{{{*}}}%
 \mu^{(N)}\right)dt %
\nonumber \\
 & = & \mathbb{E}_{\mu^{(N)}}\int_{0}^{\varepsilon}t\Bigg\{ Tr\left[\left(\frac{\operatorname{Jac}(\Delta^{(N)})}{\operatorname{Jac}(tF^{(N)}+(1-t)\tilde{F}^{(N)})}\right)^{2}\right]\label{eq:born}\\
 &  & \qquad\qquad+Tr\Big(\operatorname{Hess}(V^{(N)}(F^{(N)}+t\Delta^{(N)}))(\Delta^{N})^{2}\Big)\Bigg\} dt\nonumber \\
 & \geq & \frac{\varepsilon^{2}}{2}cN\mathbb{E}_{\mu^{(N)}}(Tr((\Delta^{(N)})^{2}).\nonumber 
\end{eqnarray}
We thus get (recalling that $D_{N}\geq0$)
\begin{eqnarray*}
\frac{c}{2}N\varepsilon^{2}\mathbb{E}_{\mu^{(N)}%
}Tr((\Delta^{(N)})^{2}) & \leq & \varepsilon(H_{V^{(N)}}(\tilde{\mu}_{V}^{(N)})-H_{V^{(N)}}(\mu_{V}^{(N)}))-\varepsilon D_{N}\\
 & \leq & \varepsilon(H_{V^{(N)}}(\tilde{\mu}_{V}^{(N)})-H_{V^{(N)}}(\mu_{V}^{(N)})).
\end{eqnarray*}
Since
\[
\alpha_{N}=\frac{1}{N^{2}}\left|H_{V^{(N)}}(\tilde{\mu}_{V}^{(N)})-H_{V^{(N)}}(\mu_{V}^{(N)})\right|\to0
\]
we can now choose $\varepsilon=\sqrt{\alpha_{N}}\to0$ to conclude
that
\begin{equation}\label{conv} %
\frac{c}{2}\mathbb{E}_{\mu^{(N)}}\left[%
\frac{1}{N}Tr((\Delta^{(N)})^{2})\right]\leq\alpha_{N}^{1/2}
\end{equation}
which completes the proof of the first point. %

Let us now assume that $W$ is $(c,A)$-convex for some $c>0$. 

\noindent
Let
$W^{(N)}=N{\rm Tr}(W(A_{1},\dots,A_{n}))$ be defined on  matrices satisfying
$\max_{j}\Vert A_{j}\Vert_{\infty}\le A'$ and be infinite otherwise. Then $W^{(N)}$ is convex.
It then follows from Caffarelli's results \cite{caffarelli:regularity,cedric}
that the optimal transport map $F^{(N)}$ taking the Gaussian measure
$\mu^{(N)}$ to the measure with density $Z_{N}^{-1}\exp(-\frac{1}{2}N\sum Tr(A_{j}^{2})+W^{(N)}(A_{1},\dots,A_{n}))$
has Jacobian uniformly bounded by $1$.

The map $(A_{1},\dots,A_{n})\to\frac{1}{N}Tr((\Delta^{(N)}(A_{1},\dots,A_{n}))^2)$
can be viewed as the composition of the map $R:(A_{1},\dots,A_{n})\to\Delta^{(N)}(A_{1},\dots,A_{n})$
and the map $Q:(A_{1},\dots,A_{n})\to\frac{1}{N}\sum Tr(A_{j}^{2})$.
The Jacobian of $R$ is given (on $(A_1,\dots,A_n)$ with $\max_j \Vert A_j \Vert < A)$ by $\mathscr{J}\mathscr{D}g-\operatorname{Hess}\Psi^{(N)}$,
where $\Psi^{(N)}$ is such that $F^{(N)}=\nabla\Psi^{(N)}.$ If $\max_{j}\Vert A_{j}\Vert_{\infty}<A$,
this is bounded (as an operator on Hilbert spaces $M_{N\times N}^{n}\to M_{N\times N}^{n}$
endowed Hilbert Schmidt norms $\Vert A\Vert=\sum_{j}Tr(A_{j}^{*}A_{j})$)
because $g$ is a power series and because of Caffarelli's bound.
Hence the map $R$ is Lipschitz with a uniform Lipschitz constant
on the set where $\max_{j}\Vert A_{j}\Vert_{\infty}<A$. The map $Q$
is Lipschitz with Lipschitz constant of the form $C/\sqrt{N}$ (see
\cite[Lemma 6.2]{alice:StFlour}). 

Therefore, by concentration inequalities and \eqref{conv}, we deduce  %
\[
\lim_{N\to\infty}\frac{1}{N}Tr((\Delta^{(N)})^{2})=0\quad\mu^{(N)}\textrm{-a.s.}
\]

We next come back to (\ref{eq:born}) and observe that $\operatorname{Jac}(\tilde{F}^{(N)})$
is bounded above uniformly by some constant $M_0$ (as a small smooth perturbation of the identity).  Furthermore,  $\operatorname{Jac}(\tilde{F}^{(N)})$
is bounded above by $1$ by \cite{caffarelli:regularity}. %
Thus we get that for all  $\varepsilon<1/2$ and some constant $M=\max(M_0,1)$ independent of $\varepsilon$,
\[
\mu^{(N)}\left[\frac{1}{N^{2}}Tr\left\{ (\operatorname{Jac}(\Delta^{(N)}))^{2}\right\} \right]\le M \alpha_{N}/\varepsilon.
\]
Taking once again $\varepsilon=\sqrt{\alpha_{N}}\to0$, we obtain
that
\[
\mu^{(N)}\left[\frac{1}{N^{2}}Tr\left\{ \left(\operatorname{Jac}(F^{(N)})-\operatorname{Jac}(\tilde{F}^{(N)})\right)^{2}\right\} \right]\to0. %
\]

We now apply \cite{caffarelli:regularity} to conclude that 
\begin{equation}
\operatorname{Jac}F^{(N)}\leq1.\label{eq:boundsOnSpectrumJacobianAtN}
\end{equation}
Let $Y_{N}=\operatorname{Jac}(F^{(N)})$, $\tilde{Y}_{N}=\operatorname{Jac}(\tilde{F}^{(N)})$
be random variables taking values in the space $M_{n\times n}(\operatorname{End}(M_{N\times N}^{sa}))\cong M_{n\times n}(M_{N\times N}\otimes M_{N\times N})$
endowed with the normalized trace $\frac{1}{n N^{2}}Tr\otimes Tr\otimes Tr$.
Both $Y_{N}$ and $\tilde{Y}_{N}$ are bounded in operator norm, and
consequently define elements $Y$ and $\tilde{Y}$ in the ultraproduct
von Neumann algebra $\prod_{N}^{\omega}M_{n\times n}(M_{N\times N}\otimes M_{N\times N})$.
Since $\Vert Y_{N}-\tilde{Y}_{N}\Vert_{2}\to0$, $Y=\tilde{Y}$. 

Once again, because of concentration,
\[
\mu^{(N)}\left[\frac{1}{N^{2}}Tr\left\{ (\operatorname{Jac}(\tilde{F}^{(N)}))^{p}\right\} \right]\to\tau\otimes\tau Tr((\mathscr{J}F)^{p})
\]
so that the spectrum of $\mathscr{J}F$ is the same as that of $\tilde{Y}$
(and so the same as that of $Y$). 

By (\ref{eq:boundsOnSpectrumJacobianAtN}), the spectrum of $Y_{N}$
lies in the interval $[0,1]$. But this implies that also the spectrum
of the limiting operator $Y$ is contained in the same set. Thus $\mathscr{J}F\leq1$.
\end{proof}

\section{Open questions.}

We list some open questions that are raised by our results.
\begin{enumerate}
\item In the classical case, Brenier's theorem \cite{brennier:polarFact,cedric}
asserts much more than the statement of our main theorem: the classical
analog of the map $\mathscr{D}g$ gives \emph{optimal }transport from
$\tau$ to $\tau_{V}$ for quadratic Wasserstein distance.  %
It would be nice to understand if the same
holds true in the non-commutative case (see \cite{bine-voiculescu:WassersteinDist}
for the extension of the notion of the Wasserstein distance to non-commutative
random variables). Note that the map we construct is optimal in the
single-variable case $n=1$.
\item Brenier gave a heuristic derivation of his theorem through a very
general ``polar factorization'' theorem. Does a theorem like that
hold in the non-commutative case? There is an infinitesimal analog
of his decomposition (related to the classical Helmholtz decomposition
of vector fields) which has been extensvively studied by Voiculescu
in \cite{dvv:cyclomorphy}. Can non-commutative monotone transport
be also obtained in the same way? 
\item Does the positivity condition $\mathscr{J}\mathscr{D}G\in \{ F> 0 : F\in M_{n\times n}(\tau\otimes\tau^{\textrm{op}})\}$ on the ``Hessian''
of $G$ translate into any kind of convexity properties of $G$?
\item What happens to our map $F=\mathscr{D}G$ in the case that $V$ is
not strictly convex? It can be seen that in the absence of the bounds
on $W$ the isomorphism (or even the embedding $C^{*}(\tau_{V})\subset C^{*}(S_{1},\dots,S_{n})$)
fails to exist at least on $C^{*}$-level. This is due to the fact
that for certain non-convex polynomials, solutions to the Schwinger-Dyson
equations may lead to $C^{*}$-algebras with non-trivial projections.
But the free semicircular system generates a projectionless $C^{*}$-algebra.
Thus failure of convexity of $V=\frac{1}{2}\sum X_{j}^{2}+W$ must
be ``visible'' as a defect of regularity of the transport map $F$.
\item Is the monotone transport map unique? More precisely, let $X_{1},\dots,X_{n}$
be a semicircular family, and assume that $Y=(Y_{1},\dots,Y_{n})$
and $Y'=(Y_{1}',\dots,Y_{n}')$ both belong to the $L^{2}$-closure
of $\{\mathscr{D}g(X_{1},\dots,X_{n}):\mathscr{J\mathscr{D}}g\geq0\}$.
If the law of $Y$ is the same as the law of $Y'$ , is $Y=Y'$?
\end{enumerate}
We can prove a uniqueness statement for our monotone transport if
we assume more on the transport map and the ``target'' $n$-tuple
$Y_{1},\dots,Y_{n}$:
\begin{thm}
Let $X_{1},\dots,X_{n}$ be a semicircular family, and assume that
$Y=(Y_{1},\dots,Y_{n})$ and $Y'=(Y_{1}',\dots,Y_{n}')$ both belong
to the $L^{2}$-closure of $\{\mathscr{D}g:\mathscr{J\mathscr{D}}g\geq0\}$.
Assume moreover that $Y$ and $Y'$ are both invertible non-commutative
power series in $X_{1},\dots,X_{n}$. If $Y$ and $Y'$ have as their
laws the same free Gibbs law $\tau_{V}$ with 
$V=\frac{1}{2}\sum X_i^2 +W$, $W$ small enough,  then $Y=Y'$. \end{thm}
\begin{proof}
The assumption that $Y,Y'$ are invertible non-commutative power series
allow us to apply Voiculescu's change of variable formula for free
entropy \cite{dvv:entropy2}:
\begin{eqnarray*}
\chi(Y) & = & \chi(X)+\tau\otimes\tau(Tr(\log\mathscr{J}Y))\\
\chi(Y') & = & \chi(X)+\tau\otimes\tau(Tr(\log\mathscr{J}Y')).
\end{eqnarray*}
Since $Y$ and $Y'$ have the same law, they have the same free entropy
(which is finite). Thus
we conclude that
\[
\tau\otimes\tau(Tr(\log\mathscr{J}Y))-\tau(V(Y))=\tau\otimes\tau(Tr(\log\mathscr{J}Y'))-\tau(V(Y')).
\]
Let
\[
\psi(Y)=\chi(Y)-\chi(X)-\tau(V(Y)). %
\]
Then $\psi$ is maximal iff $ $$Y$ has the law $\tau_{V}$ (see
e.g. \cite{dvv:entropysurvey}, we sketch the argument for completeness:
if one replaces $Y$ by $Y+\varepsilon P(Y)$ for some polynomials
$P_{1},\dots,P_{n}$, then
\[  %
\psi(Y+\epsilon P(Y))=\psi(Y)+\epsilon\left\{ \tau\otimes\tau(Tr(\mathscr{J}YP))-\tau(\mathscr{D}V(Y)P)\right\} +O(\epsilon^{2}),
\]
and so any maximizer to $\Psi$ satisfies the Schwinger-Dyson equation
and  has the same law as $Y$.)

It follows that
\[
\tau\otimes\tau(Tr(\log\mathscr{J}Y))-\tau(V(Y))=\chi(Y)-\chi(X) - \tau(V(Y))=\max_{Z}\chi(Z)-\chi(X)-\tau(V(Z)),
\]
so a fortiori

\[
\tau\otimes\tau(Tr(\log\mathscr{J}Y))-\tau(V(Y))=\sup_{Y\in\{\mathscr{D}g:\mathscr{J}\mathscr{D}g\geq0\}}\tau\otimes\tau(Tr(\log\mathscr{J}Y))-\tau(V(Y)).
\]
But since $\tau\otimes\tau(Tr(\log\mathscr{J}Y))-\tau(V(Y))$ is strictly
concave in $Y$ for $\mathscr{J}Y$ in the positive cone of $M_{n\times n}(M\bar{\otimes}M^{op})$
(here $M=W^{*}(X_{1},\dots,X_{n})$), it follows that there is at
most one $n$-tuple $Y$ in the closure of $Y\in\{\mathscr{D}g:\mathscr{J}\mathscr{D}g\geq0\}$
which gives this maximal value. Thus $Y=Y'$.
\end{proof}
\bibliographystyle{amsalpha}
\providecommand{\bysame}{\leavevmode\hbox to3em{\hrulefill}\thinspace}
\providecommand{\MR}{\relax\ifhmode\unskip\space\fi MR }
\providecommand{\MRhref}[2]{%
  \href{http://www.ams.org/mathscinet-getitem?mr=#1}{#2}
}
\providecommand{\href}[2]{#2}

\end{document}